\documentclass[12pt]{article} 
\usepackage{amsfonts,amsmath,latexsym,amssymb,mathrsfs,amsthm,comment,setspace}
\usepackage{slashbox}
\usepackage{caption}

\let\OLDthebibliography\thebibliography
\renewcommand\thebibliography[1]{
  \OLDthebibliography{#1}
  \setlength{\parskip}{0pt}
  \setlength{\itemsep}{0pt plus 0.0ex}
}

\def\numberlikeadb{\global\def\theequation{\thesection.\arabic{equation}}}
\numberlikeadb
\newtheorem{theorem}{Theorem}[section]
\newtheorem{lemma}[theorem]{Lemma}
\newtheorem{corollary}[theorem]{Corollary}

\newtheorem{proposition}[theorem]{Proposition}
\newtheorem{remark}[theorem]{Remark}

\usepackage{color} 
\newcommand{\gr}[1]{{\color{black} #1}}
\newcommand{\gdr}[1]{{\color{black} #1}}
\allowdisplaybreaks

\usepackage{lscape}
\usepackage{caption}
\usepackage{multirow}
\begin{document}

\title{Technical Report: Bounds for the chi-square approximation of Friedman's statistic by Stein's method}
\author{Robert E. Gaunt\footnote{Department of Mathematics, The University of Manchester, Oxford Road, Manchester M13 9PL, UK, robert.gaunt@manchester.ac.uk}\:\, and Gesine Reinert\footnote{Department of Statistics, University of Oxford,
24-29 St.\ Giles', Oxford OX1 3LB, UK, reinert@stats.ox.ac.uk}}

\date{} 
\maketitle

\vspace{-5mm}


\begin{abstract} Friedman's chi-square test  is a non-parametric statistical test  for $r\geq2$ treatments across $n\ge1$ trials to assess the null hypothesis that there is no treatment effect. We use Stein's method with an exchangeable pair coupling to derive an explicit bound on the distance between the distribution of Friedman's statistic  and its limiting chi-square distribution, measured using smooth test functions. Our bound is of the optimal order $n^{-1}$, and also has an optimal dependence on the parameter $r$, in that the bound tends to zero if and only if $r/n\rightarrow0$. From this bound, we deduce a Kolmogorov distance bound that decays to zero under the weaker condition $r^{1/2}/n\rightarrow0$.
\end{abstract}


\noindent{{\bf{Keywords:}}} Stein's method; Friedman's statistic; chi-square approximation; rate of convergence; exchangeable pair.

\noindent{{{\bf{AMS 2010 Subject Classification:}}} Primary 60F05; 62E17

\section{Introduction}

\subsection{Friedman's statistic and main results}

Friedman's chi-square test \cite{friedman} is a non-parametric statistical test that, given $r\geq2$ treatments across $n$ independent trials,  can be used to test the null hypothesis that there is no treatment effect against the general alternative.  Suppose that for the $i$-th trial we have the ranking $\pi_i(1),\ldots,\pi_i(r)$, where $\pi_i(j)\in\{1,\ldots,r\}$, over the $r$ treatments.  Under the null hypothesis, the rankings are independent permutations $\pi_1,\ldots,\pi_n$, with each permutation being equally likely.  Let
\begin{equation}\label{yieqn}S_j=\frac{\sqrt{12}}{\sqrt{r(r+1)n}}\sum_{i=1}^n \rho_i(j),
\end{equation}
where $\rho_i(j)=\pi_i(j)-(r+1)/2$.   Then the Friedman chi-square statistic, given by
\begin{equation}\label{miltonfa}F_r=\sum_{j=1}^rS_j^2,
\end{equation}
has mean $r-1$ and is asymptotically $\chi_{(r-1)}^2$ distributed under the null hypothesis.

The study of asymptotically chi-square distributed statistics has received much attention in the literature.  For general results and application to Pearson's statistic see \cite{gu03}, in which a $O(n^{-1})$ Kolmogorov distance bound between Pearson's statistic and the $\chi_{(r-1)}^2$ distribution was obtained for the case of $r\geq6$ cell classifications, improving a result of \cite{yarnold}.  Also,   \cite{asylbekov,
ulyanov} have used Edgeworth expansions to study the rate of convergence of the more general power divergence family of statistics constructed from the multinomial distribution of degree $r$ (which includes the Pearson, log-likelihood ratio and Freeman-Tukey statistics as special cases) to their $\chi_{(r-1)}^2$ limits.  Using Stein's method \cite{stein}, a
 $O(n^{-1})$ bound for the rate of convergence for Pearson's statistic for $r\geq2$ cell classifications was obtained by \cite{gaunt chi square}. This result has been generalised 
 recently 
 to cover family of power divergence statistics (the largest subclass for which finite sample bounds are possible) by \cite{gaunt power divergence}. Also, \cite{ar20} used Stein's method to obtain a bound with $O(n^{-1/2})$ convergence rate for the  likelihood ratio statistic when the data are realisations of independent and identically distributed random elements. Stein's method has also been used to obtain error bounds for the multivariate normal approximation of vectors of quadratic forms \cite{c08,de17}.


To date, however, Friedman's statistic has received little attention in the literature. The best result known to us is by \cite{jensen}, which provides, for any $r\geq2$, the following Kolmogorov distance bound
\begin{equation}\label{jenbound}d_{\mathrm{K}}(\mathcal{L}(F_r),\chi_{(r-1)}^2):=\sup_{z\geq0}|\mathbb{P}(F_r\leq z)-\mathbb{P}(Y_{r-1}\leq z)|\leq C(r)n^{-r/(r+1)},
\end{equation}
where the (non-explicit) constant $C(r)$ depends only on $r$, and $Y_{r-1}\sim\chi_{(r-1)}^2$. That there is little literature on distributional bounds for Friedman's statistic, and particularly no bound that has a good or even explicit dependence on $r$, may be down to the dependence structure of Friedman's statistic, which is more complicated than that of the aforementioned power divergence, Pearson and likelihood ratio statistics. To get a feel for this dependence structure, we note that whilst, for fixed $j$, the random variables $\rho_1(j),\ldots,\rho_n(j)$ are independent, the sums $S_1,\ldots,S_r$ are not independent; indeed, $S_r=-\sum_{j=1}^{r-1}S_j$.  

In this paper, we use Stein's method to obtain explicit bounds on the rate of convergence of Friedman's statistic to its limiting chi-square distribution. Stein's method is particularly well-suited to this problem, because through the use of coupling techniques it often allows one to treat even complicated dependence structures. 

  To state our error bounds,
   we need some notation. 
   Denote by $C_b^{j,k}(\mathbb{R}^+)$ the class of functions $h : \mathbb{R}^+\rightarrow\mathbb{R}$ for which $h^{(k)}$ exists and the derivatives of order $j,j+1,\ldots,k$ are bounded.
    Note that a function $h\in C_b^{j,k}(\mathbb{R}^+)$ need not be itself bounded. We denote the usual supremum norm of a function $g:\mathbb{R}^+\rightarrow\mathbb{R}$ by $\|g\|=\|g\|_\infty=\sup_{x\in \mathbb{R}^+}|g(x)|$.

Our main result is the following weak convergence theorem for smooth test functions with a bound of optimal order with respect to both parameters $n$ and $r$ for the $\chi_{(r-1)}^2$ approximation of Friedman's statistic, which holds for all $r\geq2$. 
The order $n^{-1}$ rate is the same as has been
obtained by \cite{gaunt chi square} and \cite{gaunt power divergence} for the chi-square approximation of the Pearson and power divergence statistics.  



\begin{theorem}\label{freeman} Suppose $n\geq 1$ and $r\geq 2$.  Then, for  $h \in C_b^{1,3}(\mathbb{R}^+)$, 
 \begin{equation}
 \label{thm2bound}
 |\mathbb{E}[h(F_r)]-\chi_{(r-1)}^2h| \leq
 \frac{r}{n} 
 \bigg[ \gdr{585}\|h' \|+\bigg(\gdr{2679}+\frac{431r}{n}\bigg) \|h''\|+ \bigg(\gdr{3905}+\frac{646r}{n}\bigg) \| h^{(3)} \| \bigg], 
\end{equation}
where $\chi_{(r-1)}^2h$ denotes $\mathbb{E}[h(Y_{r-1})]$ for $Y_{r-1}\sim \chi_{(r-1)}^2$. The $r/n$ rate is optimal with respect to both $n$ and $r$.
\end{theorem} 

Moreover, (\ref{betabd}) provides a more complicated bound that, for fixed $r\geq2$, has a smaller numerical value when $n\geq\gdr{293}$.


\begin{remark}In Lemma \ref{frmeanvar}, we prove that $\mathbb{E}[F_r^2]=r^2-1-2(r-1)/n$. As $ \mathbb{E}[Y_{r-1}^2]=r^2-1$ for $Y_{r-1}^2\sim \chi_{(r-1)}^2$, it follows that $|\mathbb{E}[F_r^2]-\mathbb{E}[Y_{r-1}^2]|=2(r-1)/n$. Whilst the function $h(x)=x^2$ is not in the class $C_b^{1,3}(\mathbb{R}^+)$, this lends plausibility to the optimality of the $r/n$ rate.
A rigorous justification is given in the proof of the theorem.
\end{remark}

\begin{remark}\label{20remark} To keep the calculations manageable, in proving Theorem \ref{freeman} we use some crude inequalities such as $|\rho_m(k)|\leq(r+1)/2$ and $|\rho_m(k)-\rho_m(l)|\leq(r-1)$, $k\not=l$, provided the application of these inequalities allows us to retain the $r/n$ rate. When calculations can be kept manageable and using these crude inequalities would lead to significantly worse numerical constants we argue more carefully. So as not to distract from the main thrust of the proof, which is to achieve the $r/n$ rate, we collect some of these calculations in Lemmas \ref{suplem0}--\ref{suplem3}, which are proved in Supplementary Material  \ref{appproofs}. 

As part of our proof we apply the bound (\ref{chisquarebound3}) for the third and fourth order derivatives of the solution of the $\chi_{(r-1)}^2$ Stein equation. We must apply this bound to to get the optimal dependence on $r$, but we pick up a numerical factor of 20 in doing so, for reasons which will become clear. If $r$ is relatively small and $h\in C_b^{1,4}(\mathbb{R}^+)$, applying (\ref{lukbound}) instead may give smaller numerical constants.
\end{remark}


Bounds obtained by Stein's method are often stated using smooth test functions, as in Theorem \ref{freeman}, particularly if technical issues arise in controlling solutions to the Stein equations or when faster than $O(n^{-1/2})$ convergence rates are sought; see, for example, \cite{aek20,bh85,bdf20,dp18,f18,gaunt chi square,goldstein,lefevre}. Bounds for non-smooth functions can be used for the construction of confidence intervals; however, as noted by \cite[p.\ 151]{bh85}, bounds for smooth functions may be more natural in theoretical settings, and, as is discussed in \cite[pp.\ 937--938]{goldstein}, working with smooth test functions may allow one to obtain improved error bounds that may not hold for non-smooth test functions. 


\begin{remark}\label{rem1.2}The premise of smooth test functions is crucial, because $O(n^{-1})$ bounds like that of Theorem \ref{freeman} that are valid for all $r\geq2$ will in general not hold for non-smooth test functions. To see this, consider the single point test function $h\equiv\chi_{\{0\}}$. Suppose $n=2k$ and $r=2$. Then $F_2=S_1^2+S_2^2$, where $S_2=-S_1$, so that $F_2=2S_1^2$. As $r=2$, we can write $S_1=n^{-1/2}\sum_{i=1}^nY_i$, where $Y_1,\ldots,Y_n$ are i.i.d.\ with $Y_1=\sqrt{2} \rho_1(1) 
\sim\mathrm{Unif}\{-1/\sqrt{2},1/\sqrt{2}\}$. Thus,
$$\mathbb{E}[h(F_2)]=\mathbb{P}(F_2=0)=\mathbb{P}(S_1=0)=\mathbb{P}\Big(\sum_i Y_i=0\Big)=\binom{2k}{k}\bigg(\frac{1}{2}\bigg)^{2k}\approx\frac{1}{\sqrt{\pi k}}=\sqrt{\frac{2}{\pi n}},
$$
where we used Stirling's approximation. Since $\chi_{(1)}^2h=\mathbb{P}(\chi_{(1)}^2=0)=0$, 
the Kolmogorov distance between the distribution of $F_2$ and the $\chi_{(1)}^2$ distribution cannot be of smaller order than
  $n^{-1/2}$, which is the
   order of  the Kolmogorov distance bound (\ref{jenbound}) when $r=2$.
\end{remark} 

As noted in Remark \ref{rem1.2}, in the case $r=2$ we can write $F_2=2S_1^2$, in which $S_1$ is a sum of zero mean i.i.d.\ random variables. Due to this special structure, in the following proposition we are able to obtain bounds that improve on those of Theorem \ref{freeman}, in terms of smaller numerical constants and weaker assumptions on the test functions $h$. These bounds follow easily from applying results of \cite{gaunt normal}, and the short proof is given in Section \ref{appa}. 
Recall that, for real-valued random variables $X$ and $Y$, the Wasserstein distance between their distributions is given by $d_{\mathrm{W}}(\mathcal{L}(X),\mathcal{L}(Y))=\sup_{h\in\mathcal{H}}|\mathbb{E}[h(X)]-\mathbb{E}[h(Y)]|$, where $\mathcal{H}$ is the class of Lipschitz functions with Lipschitz constant at most 1.

\begin{proposition}\label{prop1.4}Suppose $n\geq1$. Then
\begin{equation}\label{propbd1}d_{\mathrm{W}}(\mathcal{L}(F_2),\chi_{(1)}^2)\leq \frac{1}{\sqrt{n}}\bigg(87+\frac{48}{\sqrt{n}}\bigg),
\end{equation}
and, for $h\in C_b^{1,2}(\mathbb{R}^+)$,
\begin{equation}\label{propbd2}|\mathbb{E}[h(F_2)]-\chi_{(1)}^2h|\leq\frac{1}{n}\bigg(69+\frac{43}{n}\bigg)\{\|h'\|+\|h''\|\}.
\end{equation}
\end{proposition} 

By a standard argument for converting smooth test function bounds into Kolmogorov distance bounds (see \cite{chen}, p.\ 48) we deduce the following Kolmogorov distance bound from (\ref{thm2bound}). The proof is given in Section \ref{appa}. 

\begin{corollary}\label{cor1.3} Suppose $n\geq1$. Then
\begin{align}\label{kolmbd}d_{\mathrm{K}}(\mathcal{L}(F_r),\chi_{(r-1)}^2) \leq  \begin{cases} \displaystyle \frac{0.9496}{\sqrt{n}}, & \:  r=2, \vspace{2mm}  \\ \vspace{2mm} 
\displaystyle \frac{\gdr{30}}{n^{1/4}}+\frac{\gdr{75}}{n^{1/2}}+\frac{\gdr{120}}{n^{3/4}}+\frac{\gdr{8}}{n^{5/4}}+\frac{\gdr{36}}{n^{3/2}}, & \:  r=3, \\
\displaystyle \frac{\gdr{13}r^{1/8}}{n^{1/4}}\bigg(1+\frac{1}{r}\bigg)+\frac{\gdr{46}}{r^{1/4}n^{1/2}}+\frac{\gdr{54}r^{3/8}}{n^{3/4}}+\frac{\gdr{3}r^{1/8}}{n^{5/4}}+\frac{\gdr{8}r^{3/4}}{n^{3/2}}, & \:   r\geq4. \end{cases}
\end{align}
\end{corollary}


{For $r=2$, the bound in Corollary \ref{cor1.3} has an optimal dependence on $n$ (see Remark \ref{rem1.2}); we can achieve this rate by exploiting the special structure of Friedman's statistic in the case $r=2$. For $r\geq3$, t}he rate of convergence of the bound is slower than the $O(n^{-r/(r+1)})$ rate of \cite{jensen}, which is to be expected because in deriving it we apply a crude non-smooth test function approximation technique to the smooth test function bound (\ref{thm2bound}). However, to the best of our knowledge, it is the first Kolmogorov distance bound in the literature with an explicit constant involving the parameter $r$.
The bound (\ref{kolmbd}) tends to zero if $r^{1/2}/n\rightarrow0$; a weaker condition
than the $r/n\rightarrow0$ condition under which the bound (\ref{thm2bound}) tends to zero. This can be understood because the Kolmogorov distance is scale invariant, whereas for real-valued random variables $X$ and $Y$ the quantity $|\mathbb{E}[h(X)]-\mathbb{E}[h(Y)]|$ is not. From our bound (\ref{kolmbd}), we deduce that $F_r\rightarrow_d\chi_{(r-1)}^2$ if $r^{1/2}/n\rightarrow0$. To the best of our knowledge, this is the weakest condition in the literature under which Friedman's statistic is known to converge to the $\chi_{(r-1)}^2$ distribution if $r=r(n)$ is allowed to grow with $n$. We have been not been able to establish whether the condition is optimal, but at the very least it is close to being optimal.

\subsection{Discussion of methods and outline of the paper}

The proof of Theorem \ref{freeman} is long and quite involved, with delicate arguments required in particular to achieve the optimal dependence on the parameter $r$. Here we provide a summary of the key steps in the proof, which may prove useful in other problems.

First, we make a connection between the $\chi_{(r-1)}^2$ and multivariate normal Stein equations.
 Converting the problem to one of multivariate normal approximation allows us to take advantage of the powerful machinery of Stein's method for multivariate normal approximation and is natural, because Friedman's statistic is formed as a function of a random vector $\mathbf{S}$ that is asymptotically multivariate normally distributed. Whilst we convert the problem to one of multivariate normal approximation, we work with the solution of the $\chi_{(r-1)}^2$ Stein equation, so that at the end of the proof we can apply the $O(r^{-1})$ bound (\ref{chisquarebound3}) on the solution, which has an optimal dependence on $r$. 
 
In the second step, we construct an exchangeable pair coupling for Friedman's statistic, which is ideally suited to the dependence structure of Friedman's statistic. We then apply Stein's method of exchangeable pairs \cite{stein2, reinert 1} to obtain an initial bound. As we seek a bound of order $n^{-1}$, we expand one term further than is the case in the widely-used general bound of \cite{reinert 1}. The starting point for this expansion, Lemma \ref{mlml37}, is only implicitly given in \cite{reinert 1}; its explicit formulation here could be of general interest.

The third step of the proof involves bounding the remainder terms from the initial expansion. To achieve the order $n^{-1}$ rate, we apply local approach couplings to allow us to exploit the fact that $\mathbb{E}[\rho_i(j)^3]=0$ for all $i,j$, in order to vanish some remainder terms. Local couplings are less well-suited to the dependence structure of Friedman's statistic than our exchangeable pair coupling, and we need to proceed carefully to not lose the optimal dependence on $r$. One of the key steps is the introduction of the random variables $T_m=\sum_{l=1}^rS_l\rho_m(l)$, in which we absorb a sum over $r$ indices and take advantage of the fact that $S_l$ can be decomposed into a sum of $n$ random variables of which only one is dependent on $\rho_m(l)$. Thus, this paper provides a rare example for combining exchangeable pair couplings with local  couplings. The introduction of similar such random variables may also be of interest in other problems in which bounds are sought with good dependence on all parameters. 



The rest of the article is organised as follows. In Section \ref{sec2}, we introduce the necessary elements of Stein's method that we will use to prove Theorem \ref{freeman}, and make a connection between the chi-square and multivariate normal Stein equations. {In Section \ref{sec3}, we prove some preliminary lemmas needed in the proof of  Theorem \ref{freeman}; their proofs are given in Supplementary Material  \ref{appproofs}.}
 Lemma \ref{frmeanvar} provides formulas for the mean and variance of Friedman's statistic.
We prove Theorem \ref{freeman} in Section \ref{sec5}. Finally, in Section \ref{appa}, we prove Proposition \ref{prop1.4} and Corollary \ref{cor1.3}.

\section{Elements of Stein's method}\label{sec2}

In this section, we present some basic theory on Stein's method for chi-square and multivariate normal approximation.
  Originally developed for normal approximation by Charles Stein in 1972 \cite{stein}, Stein's method has since been extended to many other distributions, such as the multinomial \cite{loh}, exponential \cite{chatterjee, pekoz1}, gamma \cite{dp18,gaunt chi square, luk, nourdin1} and multivariate normal \cite{barbour2, gotze}.  For a comprehensive overview of the
 literature and an outline of the basic method  in the univariate case see for example \cite{ley}; for the multivariate case see \cite{mrs}.  

At the heart of Stein's method lies a characterising equation (a differential operator for continuous distributions) known as the Stein equation.  For the $\chi_{(p)}^2$ distribution this characterising equation is given by (see \cite{dz91,luk}):
\begin{equation}\label{chieqn}xf''(x)+\tfrac{1}{2}(p- x)f'(x)=h(x)-\chi_{(p)}^2h,
\end{equation}
where $\chi_{(p)}^2 h$ denotes the quantity $\mathbb{E}[h(Y_p)]$ for $Y_p\sim\chi_{(p)}^2$.  Evaluating both sides at a random variable of interest $W$ and taking expectations then gives
\begin{equation}\label{chieqn111}\mathbb{E}[Wf''(W)+\tfrac{1}{2}(p- W)f'(W)]=\mathbb{E}[h(W)]-\chi_{(p)}^2h.
\end{equation}
Thus, the quantity $|\mathbb{E}[h(W)]-\chi_{(p)}^2h|$ can be bounded by solving the Stein equation (\ref{chieqn}) for $f$ and then bounding the left-hand side of (\ref{chieqn111}).  It can easily be seen that 
\begin{equation}\label{bcwhc}f'(x)=\frac{\mathrm{e}^{ x/2}}{x^{p/2}}\int_0^{x}t^{p/2-1}\mathrm{e}^{-t/2}[h(t)-\chi_{(p)}^2h]\,\mathrm{d}t
\end{equation}
solves (\ref{chieqn}). For $h$ belonging to the classes $C_b^{k,k}(\mathbb{R}^+)$, $C_b^{k-1,k-1}(\mathbb{R}^+)$ and \gdr{$C_b^{k-1,k-2}(\mathbb{R}^+)$}, respectively,
the following bounds hold (see \cite{luk}, \cite{gaunt thesis}, which improves on a bound of \cite{pickett thesis}, and \cite{gaunt chi square} respectively):
\begin{eqnarray}
\label{lukbound}\|f^{(k)}\|&\leq&\frac{2}{k}\|h^{(k)}\|, \quad k\geq1, \\
\label{chisquarebound2} \|f^{(k)}\|&\leq&\bigg\{\frac{2\sqrt{\pi}+\sqrt{2}e^{-1}}{\sqrt{p+2k-2}}+\frac{4}{p+2k-2}\bigg\}\|h^{(k-1)}\|, \quad k\geq1, \\
\label{chisquarebound3} \|f^{(k)}\|&\leq&\frac{4}{p+2k-2}\big\{3\|h^{(k-1)}\|+2\|h^{(k-2)}\|\big\}, \quad k\geq2, 
\end{eqnarray} 
where $h^{(0)}\equiv h$. 
Following on from Remark \ref{20remark},  \eqref{chisquarebound3} has a $1/p$ dependence on the degrees of freedom but comes at the price of a factor of 12 for $ \|h^{(k-1)}\|$ and 8 for $\|h^{(k-2)}\|$, giving an overall numerical factor of 20 which for small $p$ may lead to numerically  larger bounds than \eqref{lukbound} would yield.

With these bounds for the solution at our disposal, we could follow the conventional approach to Stein's method and bound the distance between the distribution of Friedman's statistic $F_r$ and its limiting $\chi_{(r-1)}^2$ distribution by bounding the left-hand side of (\ref{chieqn111}).  Instead, however, we elect to follow \cite{gaunt chi square} and use the multivariate normal Stein equation in conjunction with the chi-square Stein equation.  Indeed, there is powerful array of tools for proving approximation theorems using the multivariate normal Stein equation (see \cite{chatterjee 3, goldstein 2, goldstein1, meckes, reinert 1} for coupling techniques for multivariate normal approximation).  
Let $\Sigma$ be non-negative definite.  Then the $\mathrm{MVN}(\mathbf{0},\Sigma)$ Stein equation (see \cite{goldstein1}) is 
\begin{equation} \label{mvnga} \nabla^\intercal\Sigma\nabla f(\mathbf{w})-\mathbf{w}^\intercal\nabla f(\mathbf{w})=h(\mathbf{w})-\mathbb{E}[h(\Sigma^{1/2}\mathbf{Z})].
\end{equation} 

We now obtain a connection between the Stein equations for the $\chi_{(r-1)}^2$ and $\mathrm{MVN}(\mathbf{0},\Sigma_{\mathbf{S}})$ distributions.  The lemma can be read off from Lemma 4.2 of \cite{gaunt chi square}, because the covariance matrix $\Sigma_{\mathbf{S}}$ is equal, up to a multiplicative factor, to the covariance matrix of the random vector $\mathbf{U}=(U_1,\ldots,U_r)^\intercal$ of the observed counts in Pearson's statistic under the null hypothesis of uniform classification probabilities (see Lemma 4.1 of \cite{gaunt chi square} {and Lemma \ref{robcov} below}).

\begin{lemma}\label{multsave}Let  $f \in C^2(\mathbb{R})$ and define $g: \mathbb{R}^r \to \mathbb{R}$ by $g(\mathbf{s}) = f(w)/4$ with $w = \sum_{j=1}^r s_j^2$ for $\mathbf{s} = (s_1, \ldots, s_r)^\intercal$.  Let $F_r$ and $\mathbf{S}=(S_1,\ldots,S_r)^\intercal$ be defined as in (\ref{miltonfa}) and (\ref{yieqn}).  Let $\Sigma_{\mathbf{S}}$ denote the covariance matrix of $\mathbf{S}$. 
 Then 
\begin{equation*} \mathbb{E}[\nabla^\intercal\Sigma_\mathbf{S}\nabla g(\mathbf{S})-\mathbf{S}^\intercal\nabla g(\mathbf{S})] =  \mathbb{E}[F_rf''(F_r)+\tfrac{1}{2}(r-1-F_r)f'(F_r)]. 
\end{equation*}
\end{lemma}

To bound $\mathbb{E}[\nabla^\intercal\Sigma_\mathbf{S}\nabla g(\mathbf{S})-\mathbf{S}^\intercal\nabla g(\mathbf{S})]$  we use the
exchangeable pair coupling  approach of \cite{reinert 1}.  
A pair $(X,X')$ of random variables defined on the same probability space is called \emph{exchangeable} if $\mathbb{P}(X\in B,X'\in B')=\mathbb{P}(X\in B',X'\in B)$ for all measurable sets $B$ and $B'$.  We shall use
 the following lemma, which combines equations (2.5) and (2.6)
in the proof of Theorem 2.1 of \cite{reinert 1}.  It is worth noting that up to this part of their proof the authors had only required that $\Sigma$ be non-negative definite.

\begin{lemma}\label{mlml37}Let $\mathbf{W}=(W_1,\ldots,W_d)^\intercal\in\mathbb{R}^d$.  Assume $(\mathbf{W},\mathbf{W}')$ is an exchangeable pair of $\mathbb{R}^d$--valued random vectors such that $\mathbb{E}[\mathbf{W}] = \mathbf{0}$, $\mathbb{E}[\mathbf{W}\mathbf{W}^\intercal] = \Sigma$.  Suppose further that \begin{equation}\label{kppdd}\mathbb{E}^{\mathbf{W}}[\mathbf{W}'-\mathbf{W}]=-\Lambda \mathbf{W}
\end{equation}
for an invertible $d\times d$ matrix $\Lambda$.  Then, provided $f\in C^3(\mathbb{R}^d)$, 
\begin{align} 
\gdr{\mathbb{E}[\mathbf{W}^\intercal\nabla f(\mathbf{W})]}= \frac{1}{2} \mathbb{E} [(\mathbf{W}'-\mathbf{W})^\intercal \Lambda^{-\intercal} ( \nabla f(\mathbf{W}') - \nabla f(\mathbf{W}))] . 
 \nonumber \end{align}

\end{lemma}

\section{Preliminary lemmas}\label{sec3}

Here we state some technical lemmas;
their straightforward proofs are in Supplementary Material  \ref{appproofs}.


\begin{lemma}\label{robcov} Let  $\mathbf{S}=(S_1,\ldots,S_r)^\intercal$, where the $S_j$ are defined in (\ref{yieqn}). The (non-negative definite) covariance matrix of $\mathbf{S}$, denoted by $\gr{\Sigma} = (\sigma_{jk})$, has entries
\begin{equation}\label{scov} \sigma_{jj} = \frac{r-1}{r} \quad \text{and} \quad \sigma_{jk} = -\frac{1}{r} \quad (j \neq k).
 \end{equation}
\end{lemma}

\begin{lemma}\label{lemexpformulas}Let $r\geq2$. Then, for distinct arguments,
\begin{align*}&\mathbb{E}[\rho_m(l)]=\mathbb{E}[\rho_m(l)^3]=\mathbb{E}[\rho_m(l)^2\rho_m(j)]=\mathbb{E}[\rho_m(l)\rho_m(j)\rho_m(s)]=0,\nonumber\\
&\mathbb{E}[\rho_m(l)^2]=\frac{r^2-1}{12}, \quad \mathbb{E}[\rho_m(l)\rho_m(j)]=-\frac{r+1}{12},
\end{align*}
and the following bounds hold
\begin{align*}&\mathbb{E}[\rho_m(l)^4]\leq\frac{r^4}{80}, \quad |\mathbb{E}[\rho_m(l)^3\rho_m(j)]|\leq\frac{r^3}{80},\quad
\mathbb{E}[\rho_m(l)^2\rho_m(j)^2]\leq\frac{r^4}{144}. 
\end{align*}
Also, for $r\geq3$ and $r\geq4$, respectively,
\begin{align*}
&|\mathbb{E}[\rho_m(l)^2\rho_m(j)\rho_m(s)]|\leq\frac{r^3}{144}, \quad
|\mathbb{E}[\rho_m(l)\rho_m(j)\rho_m(s)\rho_m(t)]|=\frac{(r+1)(5r+7)}{240}.
\end{align*}
Finally, suppose $r\geq2$. Then, for $j=1,\ldots,r$,
\begin{equation}\label{sj4bd}\mathbb{E}[S_j^4]\leq3-\frac{6}{5n}.
\end{equation}
\end{lemma}

\begin{lemma}\label{lemmatm}Let $T_m=\sum_{l=1}^rS_l\rho_m(l)$. Then, for $n,r\geq2$,
\begin{align}\label{tmsquare}\mathbb{E}[T_m^2]&\leq\frac{r^3}{12}\bigg(1+\frac{r}{n}\bigg), \\
\label{crt}\mathbb{E}[T_m^4]&\leq\bigg(\frac{7}{48}+\frac{r^2}{36n^2}+\frac{1}{5n}\bigg)r^6=: C_Tr^6.
\end{align}
\end{lemma}

\begin{lemma}\label{frmeanvar}Let $F_r$ denote Friedman's statistic. Then, for $n\geq1$ and $r\geq2$,
\begin{align}\label{frmoment1}\mathbb{E}[F_r]&=r-1, \\
\label{frmoment2}\mathbb{E}[F_r^2]&=r^2-1-\frac{2(r-1)}{n}, \\
\label{frvar}\mathrm{Var}(F_r)&=2(r-1)\bigg(1-\frac{1}{n}\bigg).
\end{align}
\end{lemma}

The expectations in the following lemmas are easily bounded to the correct order with respect to $r$ and $n$ using the crude inequalities $|\rho_m(k)|\leq(r+1)/2$ and $|\rho_m(k)-\rho_m(l)|\leq(r-1)$, $k\not=l$. However, with a little extra effort, bounds with better numerical constants can be obtained. 

\begin{lemma}\label{suplem0}Let $r\geq2$. Then, for $j,k\in\{1,\ldots,r\}$,
\begin{align}\label{helpfulbound2}\mathbb{E}\bigg[\bigg( \frac{ (r^2-1)}{4} \rho_m(j)  +  \rho_m(j)^3\bigg)^2\rho_m(j)^2\bigg]&\leq 0.00234r^8, \\
\label{helpfulbound3}\mathbb{E}\bigg[\bigg( \frac{ (r^2-1)}{4} \rho_m(j)  +  \rho_m(j)^3\bigg)^2\rho_m(k)^2\bigg]&\leq 0.00240r^8.
\end{align}
\end{lemma}

\begin{lemma}\label{suplem1}Let $r\geq2$. Then, for $k=1,\ldots,r$,
\begin{align}\label{lsum1}\sum_{l=1}^r\mathbb{E}[S_k^2(\rho_m(l)-\rho_m(k)\gdr{)}^4]&\leq0.1455r^5, \\
\label{lsum2}\sum_{l=1}^r\mathbb{E}[S_k^4(\rho_m(l)-\rho_m(k)\gdr{)}^4]&\leq0.6717r^5, \quad \sum_{l=1}^r\mathbb{E}[S_k^2(\rho_m(l)-\rho_m(k)\gdr{)}^6]\leq0.09116r^7.
\end{align}
\end{lemma}

\begin{lemma}\label{suplem2}Let $r\geq2$. Then, {for $j,q,t\in\{1,\ldots,r\}$},
\begin{align}\label{sup21}\mathbb{E}\big[((r^2-1)-12\rho_m(j)^2)^2\rho_m(j)^4\big]&\leq3r^8/140, \\
\label{sup22}\mathbb{E}\big[((r^2-1)-12\rho_m(j)^2)^2\rho_m(q)^4\big]&\leq 0.02440 r^8, \\
\label{sup23}\mathbb{E}\big[((r^2-1)-12\rho_m(j)^2)^2\rho_m(j)^2\rho_m(q)^2\big]&\leq 0.02292r^8, \\
\label{sup24}\mathbb{E}\big[(    (r^2-1) - 12  \rho_m(j)^2 \big)^2   \rho_m(q)^4 \rho_m(t)^4\big]&\leq 0.00111r^{12}.
\end{align}
\end{lemma}

\begin{lemma}\label{suplem3}Let $r\geq2$. Then, for $k,l\in\{1,\ldots,r\}$,
\begin{equation}\label{suplem3for}\mathbb{E}\bigg[\bigg(\frac{6}{r(r+1)}(\rho_m(k)-\rho_m(l))^2-1\bigg)^2\bigg]\leq\frac{7}{5}.
\end{equation}
\end{lemma}

\section{Proof of Theorem \ref{freeman}}\label{sec5}

{\bf{Part I: Exchangeable pair coupling and initial expansion.}} We suppose $r\geq3$, because Proposition \ref{prop1.4} provides a bound on $|\mathbb{E}[h(F_2)]-\chi_{(1)}^2h|$ that is smaller than the bound (\ref{thm2bound}) in the case $r=2$. This assumption allows us to simplify some calculations and to obtain a compact final bound; all probabilistic arguments apply equally well for all $r\geq2$ and $n\geq2$.

We begin by presenting our exchangeable pair coupling for Friedman's statistic.  Pick an index $M\in\{1,\ldots,n\}$ uniformly and \gdr{independent} indices $K,L\in\{1,\ldots,r\}$ uniformly and independently of $M$.  If $M=m$, $K=k$, $L=l$, define the permutation $\pi_m'$ by 
\begin{equation*}\pi_m'(k)=\pi_m(l), \quad \pi_m'(l) = \pi_m(k), \quad \pi_m'(i)=\pi_m(i), \quad i\not=k,l.
\end{equation*}
We then put 
\begin{eqnarray*}S_K'&=&S_K-\frac{\sqrt{12}}{\sqrt{r(r+1)n}}[\rho_M(K)-\rho_M(L)], \\
S_L'&=&S_L-\frac{\sqrt{12}}{\sqrt{r(r+1)n}}[\rho_M(L)-\rho_M(K)], \\
 S_\gdr{j}'&=&S_\gdr{j}, \quad \gdr{j}\not=K,L.
\end{eqnarray*}
It is clear that $(\mathbf{S},\mathbf{S}')$ is an exchangeable pair, and we now verify condition (\ref{kppdd}):
\begin{align*}\mathbb{E}^{\mathbf{S}}[S_\gdr{j}'-S_\gdr{j}]&=\frac{1}{r^2n}\sum_{k,l=1}^r\sum_{m=1}^n\mathbb{E}^{\mathbf{S}}[S_\gdr{j}'-S_\gdr{j}\: | \: K=k, L=l, M=m] \\
&=\frac{1}{r^2n}\frac{\sqrt{12}}{\sqrt{r(r+1)n}}\bigg\{\sum_{k=1}^r\sum_{m=1}^n\mathbb{E}^{\mathbf{S}}[\rho_m(k)-\rho_m(\gdr{j})] +\sum_{l=1}^r\sum_{m=1}^n\mathbb{E}^{\mathbf{S}}[\rho_m(l)-\rho_m(\gdr{j})]\bigg\} \\
&=-\frac{2}{rn}\frac{\sqrt{12}}{\sqrt{r(r+1)n}}\sum_{m=1}^n\mathbb{E}^{\mathbf{S}}[\rho_m(\gdr{j})] =-\frac{2}{rn}S_\gdr{j},
\end{align*}
where we used that $\sum_{k=1}^r\rho_m(k)=0$ to obtain the third equality.  Therefore condition (\ref{kppdd}) holds with $\Lambda=(2/rn)\mathrm{I}_r$, where $\mathrm{I}_r$ is the $r\times r$ identity matrix.  Hence, $\Lambda^{-\intercal} = (r n/2)  \mathrm{I}_r$.

Having established an appropriate exchangeable pair coupling, we are in a position to bound the quantity $|\mathbb{E}[h(F_r)]-\chi_{(r-1)}^2h|$.  Now, by  Lemma \ref{multsave}, 
\begin{equation*}\label{xmas12}\mathbb{E}[h(F_r)]-\chi_{(r-1)}^2h=\mathbb{E}[F_rf''(F_r)\gdr{+}\tfrac{1}{2}(r-1-F_r)f'(F_r)]=\mathbb{E}[\nabla^\intercal\Sigma\nabla g(\mathbf{S})-\mathbf{S}^\intercal\nabla g(\mathbf{S})],
\end{equation*}
where $g: \mathbb{R}^r \to \mathbb{R}$ is defined by $g(\mathbf{S}) = f(F_r)/4$ \gr{and the}
covariance matrix $\Sigma$ \gr{is given in Lemma \ref{robcov}}.
With
Lemma \ref{mlml37} and 
Taylor expansion, 
\begin{align*}
\lefteqn{\mathbb{E}[\mathbf{S}^\intercal\nabla g(\mathbf{S})]}\\
&= \frac12  \frac{rn}{2} \mathbb{E} [( \mathbf{S}' - \mathbf{S})^\intercal (\nabla g(\mathbf{S}')-  \nabla g( \mathbf{S}))]
\\
&= \frac{rn}{4} \sum_{j,u=1}^r  \mathbb{E}\bigg[ \frac{\partial^2}{\partial s_j \partial s_u} g(\mathbf{S})
\mathbb{E}^\mathbf{S} [(S_j' - S_j)(S_u' - S_u)] \bigg] \\
&\quad+\frac{rn}{8} \sum_{j,u,v=1}^r \mathbb{E}\bigg[ \frac{\partial^3}{\partial s_j \partial s_u \partial s_v } g(\mathbf{S})
\mathbb{E}^\mathbf{S} [(S_j' - S_j)(S_u' - S_u)(S_v' - S_v)]  \bigg] \\
&\quad+ \frac{rn}{24} \sum_{j,u,v,w=1}^r \mathbb{E}\bigg[ \frac{\partial^4}{\partial s_j \partial s_u \partial s_v  \partial s_w} g(\mathbf{S}_\theta^*)\mathbb{E}^\mathbf{S} [ (S_j' - S_j)(S_u' - S_u)(S_v' - S_v)(S_w' - S_w) ]
 \Big],
\end{align*}
where the $j$-th component of $\mathbf{S}_\theta^*$ is given by $\theta_jS_j+(1-\theta_j)S_j'$ for some $\theta_j=\theta_j(\mathbf{S},\mathbf{S}')\in(0,1)$. As $ \mathbb{E} [(S_j' - S_j)(S_u' - S_u)] = 4 \sigma_{ju}/(rn)$ (see Equation (2.10) in \cite{reinert 1}) we obtain that
\begin{align*}\mathbb{E}[h(F_r)]-\chi_{(r-1)}^2h=\mathbb{E}[\nabla^\intercal\Sigma\nabla g(\mathbf{S})-\mathbf{S}^\intercal\nabla g(\mathbf{S})]= R_0 +  R_1 + R_2,
\end{align*}
with 
\begin{align*}
R_0&= \frac{rn}{4} \sum_{j,u=1}^r  \mathbb{E}\bigg[ \frac{\partial^2}{\partial s_j \partial s_u} g(\mathbf{S}) \left(  \mathbb{E} [(S_j' - S_j)(S_u' - S_u)] -  
(S_j' - S_j)(S_u' - S_u) \right) \bigg], \\
R_1&=- \frac{rn}{8} \sum_{j,u,v=1}^r  \mathbb{E}\bigg[ \frac{\partial^3}{\partial s_j \partial s_u \partial s_v } g(\mathbf{S})
\mathbb{E}^\mathbf{S} [ (S_j' - S_j)(S_u' - S_u)(S_v' - S_v)]  \bigg], \\
R_2&=  -\frac{rn}{24} \sum_{j,u,v,w=1}^r \mathbb{E}\bigg[ \frac{\partial^4}{\partial s_j \partial s_u \partial s_v  \partial s_w} g(\mathbf{S}_\theta^*)  
\mathbb{E}^\mathbf{S} [ (S_j' - S_j)(S_u' - S_u)(S_v' - S_v)(S_w' - S_w) ]  \Big]. 
\end{align*} 
Recalling that $g(\mathbf{s}) =  f( h(\mathbf{s}))/4$, with $ h(\mathbf{s}) = \sum_{j=1}^r s_j^2$ so that $\frac{\partial}{\partial s_ i}  h(\mathbf{s})  = 2 s_i$ and 
$\frac{\partial^2}{( \partial s_i)^2}  h(\mathbf{s})  = 2 $ and all other second order partial derivatives and 
 all partial derivatives of $h$ of order 3 or higher vanish, we have 
\begin{align}\label{gder}
\frac{\partial^4}{\partial s_j \partial s_u \partial s_v  \partial s_w} g (\mathbf{s}) 
&=  4  f^{(4)}(h(\mathbf{s}))  s_j s_u  s_v   s_w   
+ \gdr{2} f^{(3)}(h(\mathbf{s}))  \big(  s_j s_u \mathbf{1}( v=w)  + s_j s_v \mathbf{1}( u=w)
\nonumber \\ 
&\quad+  s_j s_w \mathbf{1}( u=v)   
   + s_u s_v  \mathbf{1}( j=w)  +  s_u s_w \mathbf{1}( j=v)  + s_v s_w \mathbf{1}( j=u) 
\big)   \nonumber \\
&\quad +   f'' (h(\mathbf{s}))  \big( \mathbf{1}( v=w, j=u) +  \mathbf{1}( u=w, j=v) 
+ \mathbf{1}( u=v, j=w)   \big)
 . 
\end{align} 
The rest of the proof is devoted to bounding the remainders $R_0$, $R_1$ and $R_2$, in reverse order.


{\bf{Part II: Bounding $R_2$.}} We observe that $ (S_j' - S_j)(S_u' - S_u)(S_v' - S_v) =0$ unless at least two indices match. Moreover,  if $K=k, L=l, M=m$ then 
\begin{align} \label{vanishing} 
& \lefteqn{
(S_j' - S_j)(S_u' - S_u)(S_v' - S_v)(S_w' - S_w) 
} 
\nonumber \\
 &\quad= \frac{144}{r^2 (r+1)^2 n^2}  (\rho_m(l) - \rho_m(k) )^4    \left( {\bf{1} }\{j=u=v=w\in \{k , l\} \} \right. \nonumber\\
&\quad\quad\left.  \gr{ + {\bf{1}}
 \{  \{ j,u,v,w\}  = \{k, l\} \mbox{ with one of them appearing three times }\}} \right.\nonumber \\ 
 &\quad\quad\left.+ {\bf{1}}
 \{  \{ j,u,v,w\}  = \{k, l\} \mbox{ with each appearing twice}\} \right). 
 \end{align}
Weaving in the derivative \eqref{gder} yields
   \begin{align*}
R_2= R_{2,1} + R_{2,2} + R_{2,3},  
\end{align*} 
where
   \begin{align*}
    R_{2,1} &=   -\frac{ rn}{6} \sum_{j,u,v,w=1}^r\mathbb{E}\left[ 
  f^{(4)} \left( h ( \mathbf{S}_\theta^* )\right)  (\theta_j  S_j+ (1- \theta_j) S_j' )(\theta_u  S_u+ (1- \theta_u) S_u')\right.   \nonumber \\
&\quad \left. 
      (\theta_v  S_v+ (1- \theta_v) S_v')    (\theta_w  S_w+ (1- \theta_w) S_w') (S_j' - S_j)(S_u' - S_u)(S_v' - S_v)(S_w' - S_w)\big],\right.    \end{align*}
      and  
      \begin{align*} 
    R_{2,2} &=  -\frac{rn}{\gdr{12}} \sum_{j,u,v,w=1}^r  \mathbb{E}\left[
  f^{(3)}\left( h ( \mathbf{S}_\theta^* )\right) \left\{  (\theta_j  S_j+ (1- \theta_j) S_j' )  (\theta_u  S_u+ (1- \theta_u) S_u')  \mathbf{1}( v=w) \right.  \right. \nonumber\\
  &\quad\left. \left. +  (\theta_j  S_j+ (1- \theta_j) S_j' ) (\theta_v  S_v+ (1- \theta_v) S_v')   \mathbf{1}( u=w) 
  \right. \right. \nonumber\\
  &\quad\left. \left. 
+   (\theta_j  S_j+ (1- \theta_j) S_j' ) (\theta_w  S_w+ (1- \theta_w) S_w')  \mathbf{1}( u=v)  \right. \right. \nonumber \\
& \quad \left. \left. +   (\theta_u  S_u+ (1- \theta_u) S_u')   (\theta_v  S_v+ (1- \theta_v) S_v')   \mathbf{1}( j=w) 
\right. \right. \nonumber\\
  &\quad\left. \left. 
   +    (\theta_u  S_u+ (1- \theta_u) S_u')  (\theta_w  S_w+ (1- \theta_w) S_w')  \mathbf{1}( j=v)  
  \right. \right. \nonumber\\
  &\quad\left. \left. +  (\theta_v  S_v+ (1- \theta_v) S_v')   (\theta_w  S_w+ (1- \theta_w) S_w')  \mathbf{1}( j=u) 
\right\} 
\right. \nonumber
 \\
&\quad \left.
     (S_j' - S_j)(S_u' - S_u)(S_v' - S_v)(S_w' - S_w)   \right],\end{align*}
      and 
   \begin{align*}
    R_{2,3} &=  -\frac{rn}{24} \sum_{j,u,v,w=1}^r \mathbb{E}\big[    f''( h ( \mathbf{S}_\theta^* ))(\{\mathbf{1}( v=w, j=u) +  \mathbf{1}( u=w, j=v) \nonumber\\
    & 
\quad+ \mathbf{1}( u=v, j=w)  \} 
     (S_j' - S_j)(S_u' - S_u)(S_v' - S_v)(S_w' - S_w)   \big] .  
\end{align*}

    For the term $R_{2,1}$  we condition on $K=k, L=l, M=m$ and use \eqref{vanishing} to obtain 
     \begin{align*} 
    R_{2,1} &=   -\frac{rn}{6}  \frac{1}{r^2n} \frac{144}{r^2 (r+1)^2 n^2}  \sum_{j,u,v,w,k,l=1}^r \sum_{m=1}^n \mathbb{E}\left[ 
  f^{(4)} \left( h ( \mathbf{S}_\theta^* )\right) \right.   \nonumber \\
& \left.  \quad(\theta_j  S_j+ (1- \theta_j) S_j' )
    (\theta_u  S_u+ (1- \theta_u) S_u')  (\theta_v  S_v+ (1- \theta_v) S_v')    (\theta_w  S_w+ (1- \theta_w) S_w') \right.   \nonumber \\
& \quad\left. (\rho_m(l) - \rho_m(k) )^4    \left( {\bf{1} }\{j=u=v=w\in \{k , l\} \} \right. \right. \\
 &\quad\left.\left.  + {\bf{1}}
 \{  \{ j,u,v,w\}  = \{k, l\} \mbox{ with each appearing twice}\} \right) | K=k, L=l, M=m \right]\\
 &= R_{2,1,1} + R_{2,1,2} \gr{+ R_{2,1,3}} .
 \end{align*}
For $R_{2,1,1}$ we have
 \begin{align}
 R_{2,1,1}& 
 =   -\frac{24}{r^3(r+1)^2n^2}  \sum_{j,u,v,w,k,l=1}^r  \sum_{m=1}^n \mathbb{E}\left[ 
  f^{(4)} \left( h ( \mathbf{S}_\theta^* )\right) \right.   \nonumber \\
&\quad \left.  (\theta_j  S_j+ (1- \theta_j) S_j' )
    (\theta_u  S_u+ (1- \theta_u) S_u')  (\theta_v  S_v+ (1- \theta_v) S_v')    (\theta_w  S_w+ (1- \theta_w) S_w') \right.   \nonumber \\
&\quad \left. (\rho_m(l) - \rho_m(k) )^4  {\bf{1} }\{j=u=v=w\in \{k , l\} \} \ | K=k, L=l, M=m \right]\\
 &= -\frac{48}{r^3(r+1)^2n^2}  
    \sum_{k,l=1}^r\sum_{m=1}^n \mathbb{E}\left[ 
  f^{(4)} \left( h ( \mathbf{S}_\theta^* )\right) \right.   \nonumber \\
&\quad \left.  (\theta_k  S_k+ (1- \theta_k) S_k' )^4  (\rho_m(l) - \rho_m(k) )^4   \ | K=k, L=l, M=m \right]
\\
 &= -\frac{48}{r^3(r+1)^2n^2} 
    \sum_{k,l=1}^r  \sum_{m=1}^n \mathbb{E}\bigg[ 
  f^{(4)} ( h ( \mathbf{S}_\theta^* )) \bigg( S_k + \theta_k \frac{\sqrt{12}}{\sqrt{r(r+1)n}}  (\rho_m(l) - \rho_m(k) )  \bigg)^4   \nonumber \\
&    \quad(\rho_m(l) - \rho_m(k) )^4 
\ | K=k, L=l, M=m  \Big], \label{r211intermediate} 
\end{align}

so that, on using the basic inequality $(a+b)^4\leq 3a^4+10a^2b^2+3b^4$,
\begin{align*}
 | R_{2,1,1 }|
     &\le  \frac{ 48 \| f^{(4)}\|}{r^3 (r+1)^2 n^2}    \sum_{k,l=1}^r  \sum_{m=1}^n\bigg\{ 3\mathbb{E}[   S_k^4    (\rho_m(l) - \rho_m(k) )^4]  \\
     &\quad+\frac{10\times12}{r(r+1)n}\mathbb{E}[   S_k^2    (\rho_m(l) - \rho_m(k) )^6] +  \frac{3\times144}{r^2(r+1)^2n^2}  \mathbb{E}[  (\rho_m(l) - \rho_m(k) )^8 ]   \bigg\}.
\end{align*}
Now, by symmetry, for $r \ge 2$,
\[ 
\sum_{k,l=1}^r    (\rho_m(l) - \rho_m(k) )^8 = \sum_{k,l=1}^r    (k-l)^8 
= \frac{1}{90}r^2(r^2-1)  (2r^2 -3) (r^4 - 5 r^2 +7) \le \frac{r^{10} }{45}.
\]
With this formula and the inequalities in (\ref{lsum2}) we get that
\begin{align*}
 | R_{2,1,1 }|  &\leq \frac{ 48 \| f^{(4)}\|}{r^3 (r+1)^2 n}\bigg\{3\times 0.6717r^6+\frac{120}{r(r+1)n}\times0.09116r^8+  \frac{432}{r^2(r+1)^2n^2}  \frac{r^{10}}{45}\bigg\}\\
  &\leq\frac{ r \| f^{(4)}\|}{ n}\bigg({97}+\frac{526}{n}+\frac{461}{n^2}\bigg).
\end{align*} 
\gr{Next, 
\begin{align*}
 R_{2,1,2}& 
 =   \frac{24}{r^3(r+1)^2n^2} \sum_{j,u,v,w,k,l=1}^r  \sum_{m=1}^n \mathbb{E}\left[ 
  f^{(4)} \left( h ( \mathbf{S}_\theta^* )\right)(\theta_j  S_j+ (1- \theta_j) S_j' ) \right.   \nonumber \\
&\quad \left.  
    (\theta_u  S_u+ (1- \theta_u) S_u')  (\theta_v  S_v+ (1- \theta_v) S_v')    (\theta_w  S_w+ (1- \theta_w) S_w')(\rho_m(l) - \rho_m(k) )^4 \right.   \nonumber \\
&\quad \left.    {\bf{1}}
 \{  \{ j,u,v,w\}  = \{k, l\} \mbox{ with one of them appearing three times}\}  | K=k, L=l, M=m \right] \\
 & 
 =    \frac{192}{r^3(r+1)^2n^2}   
    \sum_{k,l=1}^r \sum_{m=1}^n \mathbb{E}\left[ 
  f^{(4)} \left( h ( \mathbf{S}_\theta^* )\right)(\theta_k  S_k+ (1- \theta_k) S_k' )^3
    (\theta_l  S_l+ (1- \theta_l) S_l') \right.   \nonumber \\
& \quad\left.    (\rho_m(l) - \rho_m(k) )^4   | K=k, L=l, M=m \right].
\end{align*}
From the basic inequality $a^3b\leq \frac{3}{4} a^4 + \frac{1}{4} b^4$ we now obtain}
 \gdr{
 \begin{align*}
 |R_{2,1,2}|&
 \le  \frac{48\| f^{(4)}\|}{r^3(r+1)^2n^2}   
    \sum_{k,l=1}^r \sum_{m=1}^n \mathbb{E}\bigg[ \bigg\{3
   \bigg( S_k + \theta_k \frac{\sqrt{12}}{\sqrt{r(r+1)n}}  (\rho_m(l) - \rho_m(k) )  \bigg)^4
     \nonumber \\
& \quad + \bigg( S_l - \theta_l \frac{\sqrt{12}}{\sqrt{r(r+1)n}}  (\rho_m(l) - \rho_m(k) )  \bigg)^4\bigg\}  (\rho_m(l) - \rho_m(k) )^4    \bigg].
 \end{align*}}
 \gdr{Bounding as we did for $R_{2,1,1}$ gives that
 \begin{align*}
 | R_{2,1,2 }| 
  &\leq\frac{ 4r \| f^{(4)}\|}{ n}\bigg({97}+\frac{526}{n}+\frac{461}{n^2}\bigg).
\end{align*} }

Also,
 \begin{align*}
 R_{2,1,\gr{3}}& 
 =   -\frac{24}{r^3(r+1)^2n^2} \sum_{j,u,v,w,k,l=1}^r  \sum_{m=1}^n \mathbb{E}\left[ 
  f^{(4)} \left( h ( \mathbf{S}_\theta^* )\right)(\theta_j  S_j+ (1- \theta_j) S_j' ) \right.   \nonumber \\
&\quad \left.  
    (\theta_u  S_u+ (1- \theta_u) S_u')  (\theta_v  S_v+ (1- \theta_v) S_v')    (\theta_w  S_w+ (1- \theta_w) S_w')(\rho_m(l) - \rho_m(k) )^4 \right.   \nonumber \\
&\quad \left.    {\bf{1}}
 \{  \{ j,u,v,w\}  = \{k, l\} \mbox{ with each appearing twice}\}  | K=k, L=l, M=m \right] \\
 & 
 =    -\frac{144}{r^3(r+1)^2n^2}   
    \sum_{k,l=1}^r \sum_{m=1}^n \mathbb{E}\left[ 
  f^{(4)} \left( h ( \mathbf{S}_\theta^* )\right)(\theta_k  S_k+ (1- \theta_k) S_k' )^2
    (\theta_l  S_l+ (1- \theta_l) S_l')^2 \right.   \nonumber \\
& \quad\left.    (\rho_m(l) - \rho_m(k) )^4   | K=k, L=l, M=m \right].
\end{align*}
From the basic inequality $ab\leq(a^2+b^2)/2$ we now obtain
\begin{align*}|  R_{2,1,\gr{3}} | &\leq\frac{144\|f^{(4)}\|}{r^3(r+1)^2n^2}\sum_{k,l=1}^r \sum_{m=1}^n\mathbb{E}\bigg[   \bigg( S_k + \theta_k \frac{\sqrt{12}}{\sqrt{r(r+1)n}}  (\rho_m(l) - \rho_m(k) )  \bigg)^2 \\
&\quad \bigg( S_l- \theta_l  \frac{\sqrt{12}}{\sqrt{r(r+1)n}}  (\rho_m(l) - \rho_m(k) )  \bigg)^2(\rho_m(l) - \rho_m(k) )^4
  \bigg]\\
  &\leq\frac{72\|f^{(4)}\|}{r^3(r+1)^2n^2}\sum_{k,l=1}^r \sum_{m=1}^n\mathbb{E}\bigg[  \bigg\{ \bigg( S_k + \theta_k \frac{\sqrt{12}}{\sqrt{r(r+1)n}}  (\rho_m(l) - \rho_m(k) )  \bigg)^4 \\
&\quad+ \bigg( S_l- \theta_l  \frac{\sqrt{12}}{\sqrt{r(r+1)n}}  (\rho_m(l) - \rho_m(k) )  \bigg)^4\bigg\}(\rho_m(l) - \rho_m(k) )^4
  \bigg].
\end{align*}
Bounding as we did for $R_{2,1,1}$ gives that
 \begin{align*}
|  R_{2,1,\gr{3} } |\leq \frac{3r\|f^{(4)}\|}{n}\bigg({97}+\frac{526}{n}+\frac{461}{n^2}\bigg).
\end{align*}
As overall bound for $R_{2,1}$ we obtain 
 \begin{align} \label{r21bound}
|  R_{2,1} | \le  \frac{r\|f^{(4)}\|}{n}\bigg({\gdr{776}}+\frac{\gdr{4202}}{n}+\frac{\gdr{3688}}{n^2}\bigg). 
\end{align}

For $R_{2,2}$, with \eqref{vanishing}, 
\begin{align*} 
    R_{2,2} 
     &=      -\frac{\gdr{12}}{r^3(r+1)^2n^2}   \sum_{j,u,v,w,k,l=1}^r  \sum_{m=1}^n  \mathbb{E}\left[
  f^{(3)}\left( h ( \mathbf{S}_\theta^* )\right) 
   \right. \nonumber\\
  &\quad\left. 
  \left\{  (\theta_j  S_j+ (1- \theta_j) S_j' )  (\theta_u  S_u+ (1- \theta_u) S_u')  \mathbf{1}( v=w) \right. \right. \nonumber \\
  &\quad\left. \left. +  (\theta_j  S_j+ (1- \theta_j) S_j' ) (\theta_v  S_v+ (1- \theta_v) S_v')   \mathbf{1}( u=w) 
  \right. \right. \nonumber\\
  &\quad\left. \left. 
+   (\theta_j  S_j+ (1- \theta_j) S_j' ) (\theta_w  S_w+ (1- \theta_w) S_w')  \mathbf{1}( u=v)  \right. \right. \nonumber \\
& \quad \left. \left. +   (\theta_u  S_u+ (1- \theta_u) S_u')   (\theta_v  S_v+ (1- \theta_v) S_v')   \mathbf{1}( j=w) 
\right. \right. \nonumber\\
  &\quad\left. \left. 
   +    (\theta_u  S_u+ (1- \theta_u) S_u')  (\theta_w  S_w+ (1- \theta_w) S_w')  \mathbf{1}( j=v)  
  \right. \right. \nonumber\\
  &\quad\left. \left. +  (\theta_v  S_v+ (1- \theta_v) S_v')   (\theta_w  S_w+ (1- \theta_w) S_w')  \mathbf{1}( j=u) 
\right\} 
 (\rho_m(l) - \rho_m(k) )^4  \right. \nonumber\\
 &\quad\left.   \left( {\bf{1} }\{j=u=v=w\in \{k , l\} \} \right. \right. \nonumber
 \\
 &\quad  + 
 {\mathbf{1}}
 \{  \{ j,u,v,w\}  = \{k, l\} \mbox{ with one of them appearing three times}\} \\
&\quad \left. \left. + {\bf{1}}
 \{  \{ j,u,v,w\}  = \{k, l\} \mbox{ with each appearing twice}\} \right)  | K=k, L=l, M=m \right] \\
 &=- \frac{12}{r^3(r+1)^2n^2}  \sum_{k,l=1}^r  \sum_{m=1}^n   \mathbb{E}\left[
  f^{(3)}\left( h ( \mathbf{S}_\theta^* )\right)\Big\{ ((2\times 6)+(2\times 6))(\theta_k  S_k+ (1- \theta_k) S_k' )^2 \right. \nonumber\\
  &\quad\left. \gdr{+(3\times 8)(\theta_k  S_k+ (1- \theta_k) S_k' )(\theta_l  S_l+ (1- \theta_l) S_l' ) } \Big\} (\rho_m(l) - \rho_m(k) )^4 | K=k, L=l, M=m  \right].
\end{align*}

With a similar reasoning as for $R_{2,1,1}$, 
\begin{align}
|  R_{2,2} | &\le \frac{{\gdr{576}}\|f^{(3)}\|}{r^3(r+1)^2n^2}\sum_{k,l=1}^r\sum_{m=1}^n\bigg\{\mathbb{E}[S_k^2(\rho_m(l)-\rho_m(k))^4]
\nonumber\\
&\quad
+\frac{12}{r(r+1)n}\mathbb{E}[(\rho_m(l)-\rho_m(k))^6]\bigg\}\nonumber\\
\label{r22bound} &\leq \frac{{\gdr{576}}\|f^{(3)}\|}{r^3(r+1)^2n}\bigg(0.1455r^6+\frac{12}{r(r+1)n}\frac{r^8}{28}\bigg)\leq\frac{r\|f^{(3)}\|}{n}\bigg({\gdr{84}}+\frac{{\gdr{247}}}{n}\bigg),
\end{align}
where in the penultimate step we used
 (\ref{lsum1}) and that
  $\sum_{k,l=1}^r(\rho_m(l)-\rho_m(k))^6\leq r^8/28$, for $r\geq2$. 

 Now, for $R_{2,3}$,   if $K=k, L=l, M=m$ then using \eqref{vanishing} gives that 
  \begin{align*}
    R_{2,3} 
     &=    -\frac{6}{r^3(r+1)^2n^2} \sum_{j,u,v,w,k,l=1}^r \sum_{m=1}^n \mathbb{E}\left[   f''\left( h ( \mathbf{S}_\theta^* )\right) \{\mathbf{1}( v=w, j=u) \right. \nonumber\\
    &\quad\left.   +  \mathbf{1}( u=w, j=v) 
+ \mathbf{1}( u=v, j=w)   \}  (\rho_m(l) - \rho_m(k) )^4
    \right.  \nonumber \\
  &\quad( {\bf{1} }\{j=u=v=w\in \{k , l\} \}\\
  &\quad+ 
 {\mathbf{1}}
 \{  \{ j,u,v,w\}  = \{k, l\} \mbox{ with one of them appearing three times}\}\\
 &\quad + {\bf{1}}
 \{  \{ j,u,v,w\}  = \{k, l\} \mbox{ with each appearing twice}\}   )\big]. 
\end{align*}
If exactly one of $\{j,u,v,w\}$ appears three times then  
$$
\{\mathbf{1}( v=w, j=u)   +  \mathbf{1}( u=w, j=v) 
+ \mathbf{1}( u=v, j=w)   \} 
=0$$ 
and hence it follows that by counting how often each combination can appear, with $2\times 3$ counts for $j=u=v=w\in \{k , l\}$ and $6\times1$ counts for the appearing twice indicators, and using that 
$\sum_{k,l=1}^r(\rho_m(l) - \rho_m(k))^4\leq r^6/15$, for $r\geq2$, gives the bound
  \begin{align} \label{r23bound} 
    | R_{2,3} |  &\le  \frac{6\|f''\|}{r^3 (r+1)^2 n} \times 12\times\frac{r^6}{15}\leq \frac{5r\|f''\|}{n}.
\end{align}

We now collate \eqref{r21bound}, \eqref{r22bound} and \eqref{r23bound} to obtain the bound
\begin{align}\label{r2bound}
| R_2 | &\le \frac{r}{n} \bigg[  \bigg({\gdr{776}}+\frac{\gdr{4184}}{n}+\frac{\gdr{3688}}{n^2}\bigg)\| f^{(4)}\|  +\bigg({\gdr{84}}+\frac{{\gdr{247}}}{n}\bigg) \| f^{(3)}\| 
+ 5  \| f''\|
\bigg].
\end{align} 

 {\bf{Part III: Bounding $R_1$.}}
 We begin by noting that if $K=k, L=l, M=m$ then 
\begin{align} \label{vanishing2} 
(S_j' - S_j)(S_u' - S_u)(S_v' - S_v)
 &= \bigg( \frac{\sqrt{12}}{\sqrt{r (r+1) n}} \bigg)^3(   \rho_m(l) - \rho_m(k) )^3 
   \big({\bf{1} }\{j=u=v\in \{k , l\} \}  \nonumber \\
&\quad -   {\bf{1}} \{ j=u=k, v=l \mbox{ or } j=v=k, u=l \mbox{ or } u=v=k, j=l\}  \nonumber \\
&  \quad+ 
 {\bf{1}} \{ j=u=l, v=k \mbox{ or } j=v=l, u=k \mbox{ or } u=v=l, j=k\}
\big).  
 \end{align} 
 By symmetry in the conditional expectation $\mathbb{E}^\mathbf{S} [ (S_j' - S_j)(S_u' - S_u)(S_v' - S_v)] $ the last two expectations of $R_1$ cancel and we have 
 \begin{align*}
R_1&=- \frac{rn}{8} \bigg( \frac{\sqrt{12}}{\sqrt{r (r+1) n}} \bigg)^3 \frac{1}{r^2 n} \sum_{k,l,j=1}^r \sum_{m=1}^n \mathbb{E}\bigg[ \frac{\partial^3}{ \partial s_j^3 } g(\mathbf{S})
  (\rho_m(l) - \rho_m(k) )^3  {\bf{1}}(j\in\{k,l\}) \bigg]\\
  &=- \frac{rn}{4} \bigg( \frac{\sqrt{12}}{\sqrt{r (r+1) n}} \bigg)^3 \frac{1}{r^2 n}  \sum_{l,j=1}^r \sum_{m=1}^n  \mathbb{E}\bigg[ \frac{\partial^3}{ \partial s_j^3 } g(\mathbf{S})
  (\rho_m(l) - \rho_m(j) )^3 \bigg].
\end{align*}
Moreover,
 \begin{align*}
 \sum_{l=1}^r   (\rho_m(l) - \rho_m(j) )^3 = \sum_{l=1}^r   \left( l - \frac{r+1}{2}  - \rho_m(j) \right)^3  = - \frac{r(r^2-1)}{4} \rho_m(j)  - r \rho_m(j)^3 ,
 \end{align*}
where we used that $\sum_{l=1}^r\big(l-(r+1)/2)\big)^3=\sum_{l=1}^r\big(l-(r+1)/2)\big)=0$.  
 Thus, 
  \begin{align*}
R_1&= \frac{1}{4} \bigg( \frac{\sqrt{12}}{\sqrt{r (r+1) n}} \bigg)^3    \sum_{j=1}^r \sum_{m=1}^n \mathbb{E}\left[ \frac{\partial^3}{ \partial s_j^3 } g(\mathbf{S})
 \left( \frac{ (r^2-1)}{4} \rho_m(j)  +  \rho_m(j)^3 \right)  \right].
   \end{align*}
   
     We note that  $\mathbb{E} [  \rho_m(j) ] =  \mathbb{E} [  \rho_m(j)^3] =0$ and that the $n$ trials are independent. 
   Now we carry out Taylor expansion around $\rho_m$ by setting $ \mathbf{S}^{(m)}$, the $r$-vector with components 
   $$ {S}^{(m)}_j = \frac{\sqrt{12}}{\sqrt{r(r+1) n}} \sum_{q \ne m} \rho_q(j) = S_j -  \frac{\sqrt{12}}{\sqrt{r(r+1) n}} \rho_m(j),$$ 
   which is independent of $\rho_m(1),\ldots,\rho_m(r)$. This gives 
     \begin{align*}
R_1&=  \frac{36}{{r^2 (r+1)^2 n^2}}    \sum_{j,k=1}^r  \sum_{m=1}^n\mathbb{E}\left[ \frac{\partial^4}{ \partial s_j^3 \partial s_k} g(\mathbf{S}_{\theta}^{(m)}
)  
 \left( \frac{ (r^2-1)}{4} \rho_m(j)  +  \rho_m(j)^3 \right) \rho_m(k)  \right],
   \end{align*} 
where the $j$-th component of $\mathbf{S}_\theta^{(m)}$ is given by $\theta_jS_j+(1-\theta_j)S_j^{(m)}$ for some $\theta_j\in(0,1)$.   Writing out
the derivative using \eqref{gder} 
gives
    \begin{align}
R_1
 &=  \frac{36}{{r^2 (r+1)^2 n^2}}    \sum_{j,k=1}^r \sum_{m=1}^n  \mathbb{E}\bigg[ 
\Big\{  4  f^{(4)} ( h( \mathbf{S}_\theta^{(m)}) ) (\theta _j S_j + (1-\theta_j)  S_j^{(m)})^3  (\theta _k S_k + (1-\theta_k)  S_k^{(m)}) \nonumber\\
&\left.  \left. 
\quad+ \gdr{6} f^{(3)} (  h( \mathbf{S}_\theta^{(m)})  ) 
 (\theta _j S_j + (1-\theta_j)  S_j^{(m)})  (\theta _k S_k + (1-\theta_k)  S_k^{(m)})(1+\mathbf{1}( j=k) )
 \right. \right.  \nonumber \\
\label{r1tobound}& \quad + 3  f'' (  h( \mathbf{S}_\theta^{(m)})  )  \mathbf{1}( j=k) 
\Big\} 
 \left( \frac{ (r^2-1)}{4} \rho_m(j)  +  \rho_m(j)^3 \right) \rho_m(k)  \bigg].
   \end{align} 
   
To bound $R_1$ we note some inequalities. Firstly, as $|\rho_m(j)|\leq (r+1)/2$, we have   
 \begin{align}\label{helpfulbound} 
{ \left| \bigg( \frac{ (r^2-1)}{4} \rho_m(j)  +  \rho_m(j)^3  \bigg)\rho_m(k)\right|  \le \frac{(r+1)^2}{4} \bigg( \frac{(r^2-1)}{4} + \frac{(r+1)^2}{4} \bigg) = \frac{r (r+1)^3}{8}.}
   \end{align}    
By independence and the H\"older and Cauchy-Schwarz inequalities, 
as $ \mathbb{E}[S_j^{(m)}] = \mathbb{E}[(S_j^{(m)})^3] = 0$, 
for $n\geq2$
\begin{align}&\mathbb{E}\big[\big(\theta _j S_j + (1-\theta_j)  S_j^{(m)}\big)^4\big]=\mathbb{E}\bigg[\bigg(S_j^{(m)} + \theta_j \frac{\sqrt{12}}{\sqrt{r(r+1) n}} \rho_m(j)\bigg)^4\bigg]\nonumber\\
&=\mathbb{E}[(S_j^{(m)})^4]+\frac{4\sqrt{12}}{\sqrt{r(r+1)n}}\mathbb{E}[(S_j^{(m)})^3]\mathbb{E}[\theta_j \rho_m(j)]+\frac{6\times12}{r(r+1)n}\mathbb{E}[(S_j^{(m)})^2]\mathbb{E}[\theta_j^2 \rho_m(j)^2]\nonumber\\
&\quad+\frac{4\times 12\sqrt{12}}{r^{3/2}(r+1)^{3/2}n^{3/2}}\mathbb{E}[S_j^{(m)}]\mathbb{E}[ \theta_j^3 \rho_m(j)^3]+\frac{144}{r^2(r+1)^2n^2}\mathbb{E}[\theta_j^4\rho_m(j)^4]\nonumber\\
&\leq 
\mathbb{E}[(S_j^{(m)})^4]+\frac{6\times12}{r(r+1)n}\mathbb{E}[(S_j^{(m)})^2]\mathbb{E}[ \rho_m(j)^2]+\frac{144}{r^2(r+1)^2n^2}\mathbb{E}[ \rho_m(j)^4]\nonumber\\
&\le   3- \frac{21}{5n} + \frac{6\times12}{r(r+1)n}\frac{(n-1)}{n}\frac{(r^2-1)}{12}+ \frac{144}{r^2(r+1)^2n^2} \frac{r^4}{80} \nonumber \\
&\le    3 + \frac{9}{5n} -\frac{21}{5n^2}
\label{4thmoment} =:A_n.
\end{align}

Here we used (\ref{sj4bd}) and the inequality $\mathbb{E}[(S_j^{(m)})^4]\leq 3-21/(5n)$, $n\geq2$, which follows from a slight modification of the argument used to obtain inequality (\ref{sj4bd}). 
Similarly,
\begin{align}&\mathbb{E}\big[\big(\theta _j S_j + (1-\theta_j)  S_j^{(m)}\big)^2\big]=\mathbb{E}\bigg[\bigg(S_j^{(m)}+ \theta_j \frac{\sqrt{12}}{\sqrt{r(r+1) n}} \rho_m(j)\bigg)^2\bigg]\nonumber\\
&=\mathbb{E}[( S_j^{(m)})^2]+\frac{2\sqrt{12}}{\sqrt{r(r+1)n}}\mathbb{E}[ S_j^{(m)}]\mathbb{E}[\theta_j \rho_m(j)]  + \frac{{12}}{{r(r+1) n}} \mathbb{E}[\theta_j^2 \rho_m(j)^2]  
\nonumber\\
&=\mathbb{E}[( S_j^{(m)})^2]+ \frac{{12}}{{r(r+1) n}} \mathbb{E}[\theta_j^2 \rho_m(j)^2]  
\le \frac{n-1}{n}+ \frac{r^2-1}{{r(r+1) n}} \leq1 . \label{2ndmomentbound}
\end{align}
 Applying inequalities (\ref{helpfulbound}), (\ref{helpfulbound2}), (\ref{helpfulbound3}) (\ref{4thmoment}) and (\ref{2ndmomentbound}) to (\ref{r1tobound}) 
together with the basic inequality $a^3b\leq 3a^4/4+b^4/4$ gives the bound
\begin{align}|R_1|&\leq\frac{36}{r^2(r+1)^2n^2}\bigg\{4\|f^{(4)}\|nr^2\times \frac{r(r+1)^3}{8}A_n\nonumber\\
&\quad+\gdr{6}\|f^{(3)}\|(nr^2\times\sqrt{0.00240}r^4+nr\sqrt{0.00234}r^4)+3\|f''\|nr\times\sqrt{0.00234}r^4\bigg\}\nonumber\\
\label{r1bound}&=\frac{18A_nr(r+1)\|f^{(4)}\|}{n}+\frac{\gdr{10.582}r^2\|f^{(3)}\|}{n}+\frac{\gdr{10.449}r\|f^{(3)}\|}{n}+\frac{5.225r\|f''\|}{n}.
\end{align}

{\bf{Part IV: Bounding $R_0$.}}
   Let us decompose $R_0$ as follows:
   \begin{align*}
R_0
&=  \frac{rn}{4} \sum_{j=1}^r \mathbb{E}\left[ \frac{\partial^2}{\partial s_j ^2} g(\mathbf{S}) \left(  \mathbb{E} [(S_j' - S_j)^2 ] -  
(S_j' - S_j)^2\right) \right] \\
&\quad+ \frac{rn}{4} \sum_{j=1}^r \sum_{u\ne j} \mathbb{E}\left[ \frac{\partial^2}{\partial s_j \partial s_u} g(\mathbf{S}) \left(  \mathbb{E} [(S_j' - S_j)(S_u' - S_u)] -  
(S_j' - S_j)(S_u' - S_u) \right) \right] \\
&= R_{0,1} + R_{0,2}. 
\end{align*} 

We first treat $R_{0,1}$. 
Note that 
by equation (2.10) in \cite{reinert 1}, $
 \mathbb{E} [(\mathbf{S}' - \mathbf{S})(\mathbf{S}' - \mathbf{S})^\intercal]  = 2 \Sigma \Lambda^\intercal = 4\Sigma/(rn)
  $
so that 
   $ \mathbb{E} [(S_j' - S_j)^2]  =  4 (r-1)/(r^2n).
  $
As
$$(S_j' - S_j)^2  =\mathbf{1}(j \in \{k, l\} ) \frac{12}{r(r+1)n} (\rho_m(k) - \rho_m(l))^2$$
it follows that
  \begin{align*}
 R_{0,1} 
& = 
 \frac{1}{r^3 (r+1) n} \sum_{j,k,l=1}^r  \sum_{m=1}^n
\mathbb{E}\bigg[ \frac{\partial^2}{\partial s_j ^2} g(\mathbf{S})  \big(  (r^2-1) -  3r 
 ( \rho_m(k) - \rho_m(l))^2{\bf{1}}\{ j \in \{k, l\} \}    \big)   
 \bigg].
\end{align*} 
Now, 
\begin{align*}
 &\sum_{k,l=1}^r \left(  (r^2-1)  -  3r 
 ( \rho_m(k) - \rho_m(l))^2{\bf{1}}\{ j \in \{k, l\} \}    \right) \\
 &= \sum_{k=1}^r   (r^2-1)  
 + \sum_{k=1}^r \sum_{l \ne k} \left(  (r^2-1)  -  3r 
 ( \rho_m(k) - \rho_m(l))^2{\bf{1}}\{ j \in \{k, l\} \}    \right) \\
 &= r^2(r^2-1) - 6 r \Big( r \rho_m(j)^2 + \frac{1}{12} r(r^2-1) \Big)= \frac12  r^2(r^2-1) - 6 r^2 \rho_m(j)^2.
\end{align*} 
Hence, expanding around  \gdr{$\mathbf{S}^{(m)}$},  there exists a 
  $ \theta=\theta (  \mathbf{S}^{(m)},   \mathbf{S}) \in (0,1)^r$  such that
\begin{align*}
R_{0,1}&=\frac{1}{2 r (r+1) n} \sum_{j,k=1}^r  \sum_{m=1}^n
\mathbb{E}\bigg[ \frac{\partial^2}{\partial s_j^2} g(\mathbf{S})  \big( 
(r^2-1) - 12  \rho_m(j)^2
 \big)   
 \bigg]\\
 & =
 \frac{1}{2 r (r+1) n} \sum_{j,k=1}^r   \sum_{m=1}^n
\mathbb{E}\bigg[ \frac{\partial^2}{\partial s_j^2} g(\mathbf{S}^{(m)})  \bigg] \mathbb{E}\big[  \big( 
(r^2-1) - 12  \rho_m(j)^2
 \big)   
 \big] \\
 &\quad+  \frac{\sqrt{12}}{2 r^{3/2} (r+1)^{3/2} n^{3/2}}    \sum_{j,k,q=1}^r  \sum_{m=1}^n 
\mathbb{E}\bigg[ \frac{\partial^3}{\partial s_j^2 \partial s_q} g(\mathbf{S}^{(m)})    \bigg]\mathbb{E}\big[  \big( 
(r^2-1) - 12  \rho_m(j)^2
 \big)   \rho_m(q)
 \big]\\
 &\quad+  \frac{3}{r^2 (r+1)^2 n^2}    \sum_{j,k,q,t=1}^r  \sum_{m=1}^n 
\mathbb{E}\bigg[ \bigg( \frac{\partial^4}{\partial s_j^2 \partial s_q \partial s_t} g(\mathbf{S}_\theta^{(m)} )   \bigg)   
 \\
&
 \quad\big( 
(r^2-1) - 12  \rho_m(j)^2
 \big)   \rho_m(q) \rho_m(t)
 \big] .
\end{align*}
By independence, the first summand vanishes.  Also,
from Lemma \ref{lemexpformulas},
 $ \mathbb{E}  [ \rho_m(q) ] = \mathbb{E}  [ \rho_m(q)^3] =0$ and $\mathbb{E} [ \rho_m(j)^2  \rho_m(q) ]=0$ for  $j \ne q$. Therefore
             the second summand in the expansion of $R_{0,1}$ also vanishes, leaving
             just the third sum involving the fourth order partial derivatives of $g$.  
With 
   \eqref{gder} 
this yields
    \begin{align*}
{R_{0,1} }
 & =  \frac{3}{r^2 (r+1)^2 n^2}
   \sum_{j,q,t=1}^r  \sum_{m=1}^n 
\mathbb{E}\Big[  \Big\{ 
4  f^{(4)} ( h ( \mathbf{S}_\theta^{(m)}  ) 
 )  \\
& \quad  ( \theta_j S_j^{(m)} + (1- \theta_j)  S_j)^2  ( \theta_q S_q^{(m)} + (1- \theta_q)  S_q )
  ( \theta S_t^{(m)} + (1- \theta_t)  S_t )
   \\
&
+  \gr{2} f^{(3)} ( h ( \mathbf{S}_\theta^{(m)} ) )  
\{ ( \theta_j S_j^{(m)} + (1- \theta_j)  S_j  )^2\mathbf{1}( q=t) 
 \\
&\quad + 2  (  \theta_j S_j^{(m)} + (1- \theta_j)  S_j  ) ( \theta_q S_q^{(m)} + (1- \theta_q)  S_q ) \mathbf{1}( j=t)   \\
&\quad 
+  2  (  \theta_j S_j^{(m)} + (1- \theta_j)  S_j )  ( \theta S_t^{(m)} + (1- \theta_t)  S_t ) \mathbf{1}(j=q)   \\
&\quad +  ( \theta_q S_q^{(m)} + (1- \theta_q)  S_q )
  ( \theta S_t^{(m)} + (1- \theta_t)  S_t \} \\
  & +   f'' (  h ( \mathbf{S}_\theta^{(m)} ) )  ( \mathbf{1}( q=t)
)  
+  2 \,   \mathbf{1}( j=q=t)  ) \Big\} (  (r^2-1) - 12  \rho_m(j)^2
 )   \rho_m(q) \rho_m(t)\Big]  
\\
&= R_{0,1,1} + R_{0,1,2} + R_{0,1,3}. 
\end{align*} 

First, 
   \begin{align*}
R_{0,1,1} 
& = \frac{12}{r^2 (r+1)^2 n^2}
   \sum_{j,q,t=1}^r  \sum_{m=1}^n 
\mathbb{E}\bigg[  
 f^{(4)} ( h (  \mathbf{S}_\theta^{(m)} ) 
 )\bigg( S_j - \theta_j \frac{\sqrt{12}}{\sqrt{r(r+1)n}}\rho_m(j)  \bigg)^2    \\
&\quad
 \bigg(S_q - \theta_q \frac{\sqrt{12}}{\sqrt{r(r+1)n}}\rho_m(q)  \bigg)
  \bigg(S_t - \theta_t \frac{\sqrt{12}}{\sqrt{r(r+1)n}}\rho_m(t)   \bigg)
 \\
&\quad 
\left(    (r^2-1) - 12  \rho_m(j)^2 \right)   \rho_m(q) \rho_m(t)\Big] \\
&= R_{0,1,1,1} + R_{0,1,1,2}.
\end{align*} 
Here, with the difference of sum notation indicating that we take the difference of the summands over the respective indices, 
\begin{align*}R_{0,1,1,1}&= \frac{12}{r^2 (r+1)^2 n^2}
   \bigg\{\sum_{j=1}^r   \sum_{q=1}^r 
 \sum_{t=1}^r \sum_{m=1}^n -\sum_{j=1}^r   \sum_{q\not=j} 
 \sum_{t\not=j,q}\sum_{m=1}^n\bigg\}
\mathbb{E}\bigg[  
 f^{(4)} ( h (  \mathbf{S}_\theta^{(m)} ) 
 )    \\
&\quad   \bigg( S_j - \theta_j \frac{\sqrt{12}}{\sqrt{r(r+1)n}}\rho_m(j)  \bigg)^2  \bigg(S_q - \theta_q \frac{\sqrt{12}}{\sqrt{r(r+1)n}}\rho_m(q)  \bigg)
\\
&\quad  
\bigg(S_t - \theta_t \frac{\sqrt{12}}{\sqrt{r(r+1)n}}\rho_m(t)   \bigg)\left(    (r^2-1) - 12  \rho_m(j)^2 \right)   \rho_m(q) \rho_m(t)\bigg]
\end{align*} 
and
\begin{align*}R_{0,1,1,2}&= \frac{12}{r^2 (r+1)^2 n^2}
 \sum_{j=1}^r   \sum_{q\not=j} 
 \sum_{t\not=j,q}\sum_{m=1}^n
\mathbb{E}\bigg[  
 f^{(4)} ( h (  \mathbf{S}_\theta^{(m)} ) 
 )    \\
&\quad   \bigg( S_j - \theta_j \frac{\sqrt{12}}{\sqrt{r(r+1)n}}\rho_m(j)  \bigg)^2  \bigg(S_q - \theta_q \frac{\sqrt{12}}{\sqrt{r(r+1)n}}\rho_m(q)  \bigg)
\\
&\quad  
\bigg(S_t - \theta_t \frac{\sqrt{12}}{\sqrt{r(r+1)n}}\rho_m(t)   \bigg)\left(    (r^2-1) - 12  \rho_m(j)^2 \right)   \rho_m(q) \rho_m(t)\bigg] . 
\end{align*} 
With $| \rho_m(k)  | \le (r+1)/2$ we can bound 
\begin{align}\label{yetanotherbound}
 |  (r^2-1)  - 12 \rho_m(k)^2 | \le 2 (r+1)(r+2), 
 \end{align}  
 and  using the basic inequality $a^2bc\leq a^4/2+b^4/4+c^4/4$ and \eqref{4thmoment} we obtain 
\begin{align*}
| R_{0,1,1,1} | 
& \leq\frac{12 \|f^{(4)}\|}{r^2(r+1)^2n^2}(r^3-r(r-1)(r-2))n\times 2(r+1)(r+2)\bigg(\frac{r+1}{2}\bigg)^2A_n\\
&{=\frac{18A_n(r+2)(r+1)(r-2/3)\|f^{(4)}\|}{r n}}\leq\frac{18A_n {r(r+3)}\|f^{(4)}\|}{n}.
\end{align*} 

We now decompose
   \begin{align*}
R_{0,1,1,2} 
= R_{0,1,1,2, 1}  + R_{0,1,1,2,2} + R_{0,1,1,2,3} 
\end{align*} 
with 
  \begin{align*}
R_{0,1,1,2,1} 
& = \frac{3}{r^2 (r+1)^2 n^2}
   \sum_{j=1}^r   \sum_{q \ne j}
 \sum_{t\ne j, q}\sum_{m=1}^n
\mathbb{E}\bigg[  
 f^{(4)} ( h ( \mathbf{S}_\theta^{(m)}  ) 
 )     \\
& \quad \bigg( S_j - \theta_j \frac{\sqrt{12}}{\sqrt{r(r+1)n}}\rho_m(j)  \bigg)^2 S_q S_t
\big(    (r^2-1) - 12  \rho_m(j)^2 \big)   \rho_m(q) \rho_m(t)\bigg], \\
R_{0,1,1,2,2} 
& =  2 \frac{3}{r^2 (r+1)^2 n^2}  \frac{\sqrt{12}}{\sqrt{r(r+1)n}}
   \sum_{j=1}^r   \sum_{q \ne j}
 \sum_{t\ne j, q}\sum_{m=1}^n
\mathbb{E}\bigg[  \theta_t 
 f^{(4)} ( h ( \mathbf{S}_\theta^{(m)} ) 
 )    \\
&\quad   \bigg( S_j - \theta_j \frac{\sqrt{12}}{\sqrt{r(r+1)n}}\rho_m(j)  \bigg)^2 S_q 
\big(    (r^2-1) - 12  \rho_m(j)^2 \big)   \rho_m(q) \rho_m(t)^2 \bigg], \\
R_{0,1,1,2,3} 
& = \frac{3}{r^2 (r+1)^2 n^2} \frac{{12}}{{r(r+1)n}}
   \sum_{j=1}^r   \sum_{q \ne j}
 \sum_{t\ne j, q}\sum_{m=1}^n
\mathbb{E}\bigg[\theta_q\theta_t  
 f^{(4)} ( h ( \mathbf{S}_\theta^{(m)} ) 
 )     \\
&  \quad\bigg( S_j - \theta_j \frac{\sqrt{12}}{\sqrt{r(r+1)n}}\rho_m(j)  \bigg)^2  
\big(    (r^2-1) - 12  \rho_m(j)^2 \big)   \rho_m(q)^2 \rho_m(t)^2 \bigg] .
\end{align*} 

The term $R_{0,1,1,2,3}$ is straightforward to bound: with the Cauchy-Schwarz inequality and inequalities (\ref{4thmoment}) and (\ref{sup24}) we obtain
\begin{align} \label{r01123bound} |R_{0,1,1,2,3}|&\leq\frac{36\|f^{(4)}\|}{r^3(r+1)^3n^3}\times nr(r-1)(r-2)\sqrt{A_n}\sqrt{0.00111}r^6\leq\frac{1.120\sqrt{A_n}r^3\|f^{(4)}\|}{n^2}.
\end{align}


 For $R_{0,1,1,2,1} $ we set $T_m=\sum_{l=1}^rS_l\rho_m(l)$.  This allows us to write
  \begin{align*}
R_{0,1,1,2,1} 
& = \frac{3}{r^2 (r+1)^2 n^2}
   \sum_{j=1}^r   \sum_{q \ne j}\sum_{m=1}^n
\mathbb{E}\bigg[  
 f^{(4)} ( h ( \mathbf{S}_\theta^{(m)}))\bigg( S_j - \theta_j \frac{\sqrt{12}}{\sqrt{r(r+1)n}}\rho_m(j)  \bigg)^2 
     \\
&  
\quad\big(    (r^2-1) - 12  \rho_m(j)^2 \big)  S_q \rho_m(q) (T_m - S_j \rho_m(j) - S_q \rho_m(q) )  \Big] \\
&=  R_{0,1,1,2,1, a} +  R_{0,1,1,2,1, b}.
\end{align*} 
Here 
 \begin{align*}
 R_{0,1,1,2,1, a} &= 
 \frac{3}{r^2 (r+1)^2 n^2}
   \sum_{j=1}^r   \sum_{q \ne j}\sum_{m=1}^n
\mathbb{E}\bigg[  
 f^{(4)} ( h ( \mathbf{S}_\theta^{(m)} ) 
 ) \bigg( S_j - \theta_j \frac{\sqrt{12}}{\sqrt{r(r+1)n}}\rho_m(j)  \bigg)^2    \\
&  \quad 
\big(    (r^2-1) - 12  \rho_m(j)^2 \big)  S_q \rho_m(q) 
T_m \Big]  \\ 
&= 
 \frac{3}{r^2 (r+1)^2 n^2}
   \sum_{j=1}^r  \sum_{m=1}^n 
\mathbb{E}\bigg[  
 f^{(4)} ( h ( \mathbf{S}_\theta^{(m)} ) 
 ) \bigg( S_j - \theta_j \frac{\sqrt{12}}{\sqrt{r(r+1)n}}\rho_m(j)  \bigg)^2    \\
&\quad   
\big(    (r^2-1) - 12  \rho_m(j)^2 \big)  
(T_m^2-S_j\rho_m(j)T_m)  \Big]  .
\end{align*}
Thus, with \eqref{4thmoment}, \eqref{yetanotherbound}  and the Cauchy-Schwarz inequality (twice) 
\begin{align*}
 | R_{0,1,1,2,1, a} |  &\le  \frac{6 (r+2) }{r^2 (r+1) n^2}    \| f^{(4)}\|    \sum_{m=1}^n  \bigg\{  r\sqrt{A_n} \sqrt{ \mathbb{E}[  T_m^4]}+  \sum_{j=1}^r \sqrt{A_n} \big( \mathbb{E}[ S_j^4 \rho_m(j)^4] \mathbb{E}[T_m^4]\big)^{1/4} 
 \bigg\} \\
 &\le  \frac{6\sqrt{A_n} (r+2) \| f^{(4)}\|}{r (r+1) n} \bigg\{\sqrt{C_T}r^3+3^{1/4}\frac{r+1}{2}C_T^{1/4}r^{3/2}   \bigg\} \\
& \leq \frac{12\sqrt{A_n}r(r+2)\|f^{(4)}\|}{n}\bigg(\sqrt{C_T}+\frac{1}{2\cdot3^{1/4}}C_T^{1/4}\bigg)\\
&\leq\frac{\sqrt{A_n}r(r+2)\|f^{(4)}\|}{n}\bigg(2.28+14.28\bigg(\frac{7}{48}+\frac{1}{5n}\bigg)^{1/2}+\frac{2.38r}{n}\bigg),
 \end{align*}
where we used \eqref{crt}, that \gdr{$\mathbb{E}[ S_j^4 ] < 3$}, that $\rho_m(j)^4 \le  ((r+1)/2)^4$, that  $r\geq3$ and the inequality $C_T^{1/4}\leq(1+\sqrt{C_T})/2$ and that $\sqrt{C}_T\leq(7/48+1/(5n))^{1/2}+r/(6n)$.  Also, making similar considerations,
 \begin{align*}
 |R_{0,1,1,2,1, b}|
 &= \bigg|- \frac{3}{r^2 (r+1)^2 n^2}
   \sum_{j=1}^r   \sum_{q \ne j}\sum_{m=1}^n
\mathbb{E}\bigg[  
 f^{(4)} ( h ( \mathbf{S}_\theta^{(m)}) 
 )  \bigg( S_j - \theta_j \frac{\sqrt{12}}{\sqrt{r(r+1)n}}\rho_m(j)  \bigg)^2   \\
&\quad   
\big(    (r^2-1) - 12  \rho_m(j)^2 \big)  S_q \rho_m(q)  ( S_j \rho_m(j) + S_q \rho_m(q) )  \Big] \Big|\\
&\leq \frac{6(r+2)\|f^{(4)}\|}{r^2 (r+1) n^2}nr(r-1)\sqrt{A_n}\Big\{\big(\mathbb{E}[S_1^4\rho_1(1)^4]\mathbb{E}[S_2^4\rho_1(2)^4]\big)^{1/4}+\sqrt{\mathbb{E}[S_2^4\rho_1(2)^4]}\Big\}\\
&\leq\frac{6\sqrt{A_n}(r+2)(r-1)\|f^{(4)}\|}{r(r+1)n}\times 2\sqrt{3}\bigg(\frac{r+1}{2}\bigg)^2\leq\frac{3\sqrt{3}{\sqrt{A_n}}r(r+2)\|f^{(4)}\|}{n}.
\end{align*} 
Thus,
\begin{align}\label{R01121bound}
 |R_{0,1,1,2,1}| & \leq\frac{\sqrt{A_n}r(r+2)\|f^{(4)}\|}{n}\bigg(7.48+14.28\bigg(\frac{7}{48}+\frac{1}{5n}\bigg)^{1/2}+\frac{2.38r}{n}\bigg).
\end{align}

Next, using $T_m$ again,
  \begin{align*}
R_{0,1,1,2,2} 
&=  \frac{12\sqrt{3}}{r^{5/2}(r+1)^{5/2}n^{5/2}}
   \sum_{j=1}^r   \sum_{t \ne j}\sum_{m=1}^n
\mathbb{E}\bigg[  \theta_t 
 f^{(4)} ( h ( \mathbf{S}_\theta^{(m)} ) 
 )  \bigg( S_j - \theta_j \frac{\sqrt{12}}{\sqrt{r(r+1)n}}\rho_m(j)  \bigg)^2   \\
&\quad   
\big(    (r^2-1) - 12  \rho_m(j)^2 \big)    \rho_m(t)^2  (T_m  - S_j \rho_m(j) - S_t \rho_m(t)) \Big] .
\end{align*} 
Using the Cauchy-Schwarz inequality, \eqref{4thmoment}, \eqref{tmsquare}, that $\mathbb{E}[  S_j^2] \le 1$ and \eqref{yetanotherbound} we obtain 
 \begin{align*}
| R_{0,1,1,2,2} | &\leq\frac{12\sqrt{3}\|f^{(4)}\|}{r^{5/2}(r+1)^{5/2}n^{5/2}}\times nr(r-1)\times 2(r+1)(r+2)\bigg(\frac{r+1}{2}\bigg)^2
\\
&\quad\
\sqrt{A_n}\bigg\{\sqrt{\mathbb{E}[T_1^2]}+2\times\frac{r+1}{2}\sqrt{\mathbb{E}[S_1^2]}\bigg\}\\
&\leq\frac{6\sqrt{3}\sqrt{A_n}r(r+2)\|f^{(4)}\|}{(r+1)n^{3/2}}\bigg\{\frac{r^{3/2}}{\sqrt{12}}\bigg(1+\frac{r}{n}\bigg)^{1/2}+(r+1)\bigg\}\\
&
\leq\frac{3\sqrt{A_n}r(r+2)\|f^{(4)}\|}{2n}\bigg(3+\frac{5r}{n}\bigg),
\end{align*}
employing that $(a+b)^{1/2}\leq \sqrt{a}+\sqrt{b}$ and that $r\geq3$
in the penultimate step.
The last bound is crude but chosen for easier combination with other bounds.
Combining this bound with \eqref{r01123bound} and \eqref{R01121bound}, 
 \begin{align*}
| R_{0,1,1,2} | 
&\leq 
 \frac{\sqrt{A_n}r(r+2)\|f^{(4)}\|}{n}\bigg(11.98+14.28\bigg(\frac{7}{48}+\frac{1}{5n}\bigg)^{1/2}+\frac{11r}{n}\bigg)
\end{align*}
and 
 \begin{align*} 
| R_{0,1,1} | 
&\leq 
 \frac{\sqrt{A_n}r(r+3)\|f^{(4)}\|}{n}\bigg( 18 \sqrt{A_n} + 11.98+14.28\bigg(\frac{7}{48}+\frac{1}{5n}\bigg)^{1/2}+\frac{11r}{n}\bigg) . 
\end{align*}


We now bound  $R_{0,1,2}$. By  symmetry, 
   \begin{align*}
R_{0,1,2} 
 &=  \frac{\gdr{24}}{r^2 (r+1)^2 n^2}
   \sum_{j,q=1}^r    \sum_{m=1}^n
\mathbb{E}\Big[ 
  f^{(3)} ( h ( \mathbf{S}_\theta^{(m)} ) )  
 \Big\{( \theta_j S_j^{(m)} + (1- \theta_j)  S_j  )^2  \rho_m(q)^2 \\
 &\quad+4(  \theta_j S_j^{(m)} + (1- \theta_j)  S_j  ) ( \theta_q S_q^{(m)} + (1- \theta_q)  S_q )\rho_m(q) \rho_m(j)\Big\}\big(  (r^2-1) - 12  \rho_m(j)^2
 \big)   \Big] \\
 &\quad + 
 \frac{\gdr{24}}{r^2 (r+1)^2 n^2}
   \sum_{j,q,t=1}^r  \sum_{m=1}^n 
\mathbb{E}\Big[  
  f^{(3)} ( h ( \mathbf{S}_\theta^{(m)} ) )  
 ( \theta_q S_q^{(m)} + (1- \theta_q)  S_q )  \\
&\quad( \theta S_t^{(m)} + (1- \theta_t)  S_t )
\big(  (r^2-1) - 12  \rho_m(j)^2
 \big)   \rho_m(q) \rho_m(t)\Big]  .
\end{align*} 
As $ 12 \sum_{j=1}^r \rho_m(j)^2 =  r (r^2-1)$ the last summand vanishes and  with   the Cauchy-Schwarz inequality as well as inequalities (\ref{4thmoment}), (\ref{sup22}) and (\ref{sup23}),
   \begin{align*} 
| R_{0,1,2}  | 
 & \le    \frac{\gdr{24}\|f^{(3)}\| }{r^2 (r+1)^2 n^2} \times nr^2\sqrt{A_n}\big(\sqrt{0.02440}r^4+4\sqrt{0.02292}r^4\big)
\leq \frac{\gdr{18.283}\sqrt{A_n}r^2\|f^{(3)}\|}{n}.
\end{align*} 

For $ R_{0,1,3}$, as {$ \sum_{j=1}^r ((r^2-1) - 12  \rho_m(j)^2 )=  0$},
we have with inequality (\ref{sup21}),
      \begin{align*}
|R_{0,1,3}| 
 & =   \frac{3}{r^2 (r+1)^2 n^2}  \bigg|\sum_{j=1}^r  \sum_{m=1}^n 
\mathbb{E}\left[    f'' (  h (\mathbf{S}_\theta^{(m)}  ) )   
\left(  (r^2-1) - 12  \rho_m(j)^2
 \right)   \rho_m(j)^2 \right] \bigg|\nonumber\\
 &\le  \frac{3\| f''\| }{r^2 (r+1)^2 n^2} \times r n \times \sqrt{\frac{3}{140}}r^4\leq\frac{0.440r\|f''\|}{n}. 
\end{align*} 

Having bounded $R_{0,1}$ to the desired order we move on to $R_{0,2}$.
For $u \ne j$ again by equation (2.10) in \cite{reinert 1} 
 $
\mathbb{E} [(S_j' - S_j)(S_u' - S_u)] = 4\Sigma_{ju}/(rn) = - 4/(r^2n) . 
$
Hence, 
     \begin{align*}
     R_{0,2} &= 
\frac{rn}{4} \sum_{j=1}^r \sum_{u\ne j} \mathbb{E}\left[ \frac{\partial^2}{\partial s_j \partial s_u} g(\mathbf{S}) \left(  \mathbb{E} [(S_j' - S_j)(S_u' - S_u)] -  
(S_j' - S_j)(S_u' - S_u) \right) \right] \\
%
&= \frac{1}{4r}  \sum_{j,k,l=1}^r \sum_{u\ne j}  \sum_{m=1}^n  \mathbb{E}\left[ \frac{\partial^2}{\partial s_j \partial s_u} g(\mathbf{S})  
 \right. \\
&\quad\left.
 \left(   - \frac{4}{r^2n} 
+ \frac{12}{r(r+1)n} (\rho_m(k) - \rho_m(l))^2  {\bf{1}}\{ \{j,u\} = \{ k,l\} \}   \right) \right] \\
&= \frac{1}{r}   \sum_{k=1}^r \sum_{ l \ne k } \sum_{m=1}^n  \mathbb{E}\left[ \frac{\partial^2}{\partial s_k \partial s_l} g(\mathbf{S}) \left( -  \frac{1}{ n } 
+  \frac{6}{r(r+1)n} (\rho_m(k) - \rho_m(l))^2\right) \right].
\end{align*} 
Again with Taylor expansion and the independence of the
 trials, for some $ \theta=\theta (  \mathbf{S}^{(m)},   \mathbf{S}) \in (0,1)^r$, 
   \begin{align*}
     R_{0,2} &= 
 \frac{1}{r}   \sum_{k=1}^r \sum_{ l \ne k} \sum_{m=1}^n  \mathbb{E}\left[ \frac{\partial^2}{\partial s_k \partial s_l} g(\mathbf{S}^{(m)} )   \right] \mathbb{E} \left[  \left(   -  \frac{1}{ n } 
+  \frac{6}{r(r+1)n} (\rho_m(k) - \rho_m(l))^2 \right) \right]
 \\
&\quad + \frac{1}{r}  \frac{\sqrt{12}}{\sqrt{r(r+1) n}}  \sum_{k,q=1}^r \sum_{ l\ne k} \sum_{m=1}^n   \mathbb{E}\left[ \frac{\partial^3}{\partial s_k \partial s_l \partial s_q } g(\mathbf{S}^{(m)} )  \right]
 \\
&  \quad
\mathbb{E} \left[  \left(  -  \frac{1}{ n } 
+  \frac{6}{r(r+1)n} (\rho_m(k) - \rho_m(l))^2\right)  \rho_m(q) \right] \\
&\quad+  \frac{1}{r}  \frac{6}{{r(r+1) n}}  \sum_{k,q,t=1}^r \sum_{ l\ne k } \sum_{m=1}^n   \mathbb{E}\left[ \frac{\partial^4}{\partial s_k \partial s_l \partial s_q  \partial s_t} g(  \mathbf{S}_\theta^{(m)} ) 
\right. \\
&\quad\left. 
 \left(   -  \frac{1}{ n } 
+  \frac{6}{r(r+1)n} (\rho_m(k) - \rho_m(l))^2 \right)  \rho_m(q) \rho_m(t) \right].
\end{align*} 
Again the first term vanishes by independence and the second term vanishes  similarly as  for $R_{0,1}$. 
Hence, with \eqref{gder},  
 \begin{align*}
     R_{0,2} &= 
\frac{{6}}{{r^2(r+1) n^2}}  \sum_{k,q,t=1}^r \sum_{ l\ne k} \sum_{m=1}^n  
 \mathbb{E}\bigg[ \big\{  4  f^{(4)}(  h (  \mathbf{S}_\theta^{(m)}  )  )   \{  \theta_k S_k^{(m)} + (1- \theta_k)  S_k) 
 \\
&\quad ( \theta_l S_l^{(m)} + (1- \theta_l)  S_l )  ( \theta_q S_q^{(m)} + (1- \theta_q)  S_q )  ( \theta_t S_t^{(m)} + (1- \theta_t)  S_t )
   \\
&
+\gdr{2} f^{(3)} (  h ( \mathbf{S}_\theta^{(m)}  ) )  \big[  ( \theta_k S_k^{(m)} + (1- \theta_k)  S_k)  ( \theta_l S_l^{(m)} + (1- \theta_l)  S_l ) \mathbf{1}( q=t) \\
& \quad +  ( \theta_k S_k^{(m)} + (1- \theta_k)  S_k)  ( \theta_q S_q^{(m)} + (1- \theta_q)  S_q )  \mathbf{1}( l=t) \\
& 
\quad+   ( \theta_k S_k^{(m)} + (1- \theta_k)  S_k)   ( \theta_t S_t^{(m)} + (1- \theta_t)  S_t )\mathbf{1}( l=q)    \\
&  \quad +  ( \theta_l S_l^{(m)} + (1- \theta_l)  S_l ) ( \theta_t S_t^{(m)} + (1- \theta_t)  S_t )  \mathbf{1}( j=q)  \\
& \quad+  ( \theta_l S_l^{(m)} + (1- \theta_l)  S_l )  ( \theta_t S_t^{(m)} + (1- \theta_t)  S_t ) \mathbf{1}( k=q)  
\big] \\
&
+   f'' (  h (  \mathbf{S}_\theta^{(m)}  ) )  (  \mathbf{1}( k=q,l=t) 
+ \mathbf{1}( k=t, l=q)   )
\big\}  
\\
&\quad
\bigg( -  1
+  \frac{6}{r(r+1)} (\rho_m(k) - \rho_m( l))^2\bigg) \rho_m(q) \rho_m(t) \bigg]\\
&= R_{0,2,1} + R_{0,2,2} + R_{0,2,3}.
\end{align*} 

For $R_{0,2,1}$ we employ $T_m$  from Lemma \ref{lemmatm} and write
 \begin{align*}
     R_{0,2,1} 
&= \frac{24}{{r^2(r+1) n^2}}  \sum_{ l=1}^r  \sum_{k\ne l }  \sum_{m=1}^n  
 \mathbb{E}\bigg[ 
 f^{(4)} ( h (  \mathbf{S}_\theta^{(m)} ) 
 ) \bigg(  S_k  - \theta_q \frac{\sqrt{12}}{\sqrt{r(r+1)n}} \rho_m(k) \bigg)    \\
&  
\quad  
  \bigg(  S_l - \theta_l \frac{\sqrt{12}}{\sqrt{r(r+1)n}} \rho_m(l)  \bigg) \bigg(  T_m  - \sum_{q=1}^r  \theta_q \frac{\sqrt{12}}{\sqrt{r(r+1)n}} \rho_m(q)^2 \bigg)    \\
&\quad  \bigg( T_m   -  \sum_{t=1}^r  \theta_t \frac{\sqrt{12}}{\sqrt{r(r+1)n}} \rho_m(t)^2  \bigg)
 \bigg( -  1
+  \frac{6}{r(r+1)} (\rho_m(k) - \rho_m( l))^2\bigg)  \bigg] \\
&= \frac{24}{{r^2(r+1) n^2}}  \sum_{ l=1}^r  \sum_{k\ne l }  \sum_{m=1}^n  
 \mathbb{E}\bigg[ 
 f^{(4)} ( h (  \mathbf{S}_\theta^{(m)} ) 
 )\bigg(  S_k  - \theta_q \frac{\sqrt{12}}{\sqrt{r(r+1)n}} \rho_m(k) \bigg)     \\
&\quad 
  \bigg(  S_l - \theta_l \frac{\sqrt{12}}{\sqrt{r(r+1)n}} \rho_m(l)  \bigg) \bigg(  T_m^2   -  2 T_m \frac{\sqrt{12}}{\sqrt{r(r+1)n}} \sum_{q=1}^r  \theta_q  \rho_m(q)^2    \\
&\quad
+  \frac{{12}}{{r(r+1)n}}   \bigg( \sum_{t=1}^r  \theta_t \rho_m(t)^2  \bigg)^2 \bigg) 
    \bigg( -  1
+  \frac{6}{r(r+1)} (\rho_m(k) - \rho_m( l))^2\bigg)    \bigg] 
. 
\end{align*} 
We now note that, since $|\rho_m(k) - \rho_m(l)| \le r-1$, we have
\begin{align*}\bigg|-  1
+  \frac{6}{r(r+1)} (\rho_m(k) - \rho_m( l))^2\bigg|\leq\frac{6(r-1)^2}{r(r+1)}-1<5.
\end{align*}
Using this bound, the Cauchy-Schwarz inequality as well as  \eqref{tmsquare}, \eqref{crt} and \eqref{4thmoment}, and using that $r\geq3$ in the final step, gives that 
\begin{align*}|R_{0,2,1}|&\leq\frac{24\|f^{(4)}\|}{r^2(r+1)n^2}\times 5r(r-1)n\times \sqrt{A_n}\bigg\{\sqrt{\mathbb{E}[T_1^4]}+\frac{4\sqrt{3}}{\sqrt{r(r+1)n}}\frac{r(r+1)}{2}\sqrt{\mathbb{E}[T_1^2]}\\
&\quad+\frac{12}{r(r+1)n}\frac{r^2(r+1)^2}{4}\bigg\}\\
&\leq\frac{120\sqrt{A_n}(r-1)\|f^{(4)}\|}{r(r+1)n}\bigg\{\sqrt{C_T}r^3+\frac{2\sqrt{3}\sqrt{r(r+1)}}{\sqrt{n}}\frac{r^{3/2}}{\sqrt{12}}\bigg(1+\frac{r}{n}\bigg)^{1/2}+\frac{3r(r+1)}{n}\bigg\}\\
&\leq\frac{120\sqrt{A_n}r^2\|f^{(4)}\|}{n}\bigg\{\sqrt{C_T}+\frac{r-1}{r\sqrt{(r+1)n}}\bigg(1+\sqrt{\frac{r}{n}}\bigg)+\frac{3(r-1)}{r^2n}\bigg\}\\
&\leq \frac{120\sqrt{A_n}r^2\|f^{(4)}\|}{n}\bigg\{\sqrt{C_T}+\frac{3\sqrt{5}}{20\sqrt{n}}+\frac{5}{3n}\bigg\}\\
&\leq\frac{\sqrt{A_n}r^2\|f^{(4)}\|}{n}\bigg\{120\bigg(\frac{7}{48}+\frac{1}{5n}\bigg)^{1/2}+\frac{20r}{n}+\frac{18\sqrt{5}}{\sqrt{n}}+\frac{200}{n}\bigg\}.
\end{align*}

 For $R_{0,2,2}$ we use the Cauchy-Schwarz inequality and \eqref{4thmoment} and (\ref{suplem3for}) to obtain
\begin{align*}|R_{0,2,2}|&\leq\frac{\gdr{12}\times5\|f^{(3)}\|}{r^2(r+1)n^2}\times nr^3\sqrt{A}_n\times \sqrt{\frac{7}{5}}\bigg(\frac{r+1}{2}\bigg)^2\leq\frac{\gdr{17.749}\sqrt{A_n}r(r+1)\|f^{(3)}\|}{n}.
\end{align*}
Bounding $R_{0,2,3}$ is also simple:
\begin{align*}|R_{0,2,3}|\leq\frac{6\times2\|f''\|}{r^2(r+1)n^2}\times nr^2\times \sqrt{\frac{7}{5}}\bigg(\frac{r+1}{2}\bigg)^2=\frac{3.550(r+1)\|f''\|}{n}\leq\frac{4.733r\|f''\|}{n},
\end{align*}
where in the final step we used that $r\geq3$.

Summing up our bounds for $|R_{0,1,1}|$, $|R_{0,1,2}|$, $|R_{0,1,3}|$, $|R_{0,2,1}|$, $|R_{0,2,2}|$, $|R_{0,2,3}|$
yields the following bound for $R_{0}$:
\begin{align}|R_{0}|&\leq \frac{\sqrt{A_n}r(r+3)\|f^{(4)}\|}{n}\bigg[18\sqrt{A_n}+11.98+134.28\bigg(\frac{7}{48}+\frac{1}{5n}\bigg)^{1/2}+\frac{31r}{n}+\frac{18\sqrt{5}}{\sqrt{n}}+\frac{200}{n}\bigg]\nonumber\\
\label{r0bound}&\quad+\frac{\gdr{36.032}\sqrt{A_n}(r+1)r\|f^{(3)}\|}{n}+\frac{5.173r\|f''\|}{n}.
\end{align} 
Finally, adding the bounds \eqref{r2bound},  \eqref{r1bound} and \eqref{r0bound} for $R_2$, $R_1$ and $R_{0}$ yields the bound
\begin{equation*}|\mathbb{E}[h(F_r)]-\chi_{(r-1)}^2h|\leq\frac{r}{n} \big\{ (\alpha_1+\alpha_2) \|f^{(4)} \| + (\alpha_3+\alpha_4) \|f^{(3)} \|+ \alpha_5 \|f'' \|\big\},
\end{equation*}
where 
\begin{align*}
\alpha_1 &= \gdr{776}+ \frac{\gdr{4202}}{n} + \frac{\gdr{3688}}{n^2}, \quad \alpha_2=  (r+3)\sqrt{A_n}\bigg(B_n+\frac{31r}{n}\bigg), \\
\alpha_3 &= \gdr{94.449}+ \frac{\gdr{247}}{n}, \quad
\alpha_4= (r+1)\big( \gdr{10.582} + \gdr{36.032} \sqrt{A_n}\big),
\\
\alpha_5 &= 5 + 5.5225+ 0.44 + 4.733 = 15.398,
\end{align*} 
and
\[B_n= 36\sqrt{A_n}+ 11.98 + 134.28 \bigg( \frac{7}{48} + \frac{1}{5n}\bigg)^{1/2} + 
\frac{18 \sqrt{5}}{\sqrt{n}} + \frac{200}{n} .\]

We can use (\ref{chisquarebound2}) to bound $\|f^{(4)}\|\leq \{(2\sqrt{\pi}+\sqrt{2}\mathrm{e}^{-1})/\sqrt{8}+4/8\}\|h^{(3)}\|\leq1.9373\|h^{(3)}\|$, since $r\geq3$. We use this inequality to bound the term $\alpha_1\|f^{(4)}\|$. By
 (\ref{chisquarebound3}),
 $\|f^{(k)}\|\leq4(3\|h^{(k-1)}\|+2\|h^{(k-2)}\|)/(r+3)$, $k=3,4$, and we apply this inequality to bound $\alpha_2\|f^{(4)}\|$ and $\alpha_4\|f^{(3)}\|$. We bound the terms $\alpha_3\|f^{(3)}\|$ and $\alpha_5\|f''\|$ using (\ref{lukbound}). This yields that, for $n\geq2$,
\begin{equation}\label{betabd}|\mathbb{E}[h(F_r)]-\chi_{(r-1)}^2h|\leq\frac{r}{n} \big\{\beta_1(n)\|h' \|  + \beta_2(n,r) \|h'' \|+\beta_3(n,r)\|h^{(3)} \| \big\},
\end{equation}
where 
\begin{align*}
\beta_1(n)&=8(\gdr{10.582}+\gdr{36.032})\sqrt{A_n} =\gdr{84.656}+\gdr{288.256}\sqrt{A_n}, \\
\beta_2(n,r)&=15.398+12(\gdr{10.582}+\gdr{36.032}\sqrt{A_n})+8\sqrt{A_n}\bigg(B_n+\frac{31r}{n}\bigg)\\
&=\gdr{142.382}+\gdr{432.384}\sqrt{A_n}+8\sqrt{A_n}\bigg(B_n+\frac{31r}{n}\bigg),\\
\beta_3(n,r)&=\frac{2}{3}\bigg(\gdr{94.449}+\frac{\gdr{247}}{n}\bigg)+1.9373\bigg(\gdr{776}+\frac{\gdr{4202}}{n}+\frac{\gdr{3688}}{n^2}\bigg)+12\sqrt{A_n}\bigg(B_n+\frac{31r}{n}\bigg)\\
&=\gdr{\gdr{1202}.37}+\frac{\gdr{\gdr{6338}.75}}{n}+\frac{\gdr{5416.39}}{n^2}+12\sqrt{A_n}\bigg(B_n+\frac{31r}{n}\bigg).
\end{align*}

Lastly, we deduce the compact final bound (\ref{thm2bound}). To do so, we first note the following simple bound. Letting $Y_{r-1}\sim\chi_{(r-1)}^2$, by the mean value theorem and (\ref{frmoment1}) we have that, for all $n\geq1$,
\begin{equation}\label{simplebd}|\mathbb{E}[h(F_r)]-\chi_{(r)}^2h|\leq\|h'\|\mathbb{E}|F_r-Y_{r-1}|\leq\|h'\|(\mathbb{E}[F_r]+\mathbb{E}[Y_{r-1}])=2(r-1)\|h'\|.
\end{equation}
We now observe that $\beta_1(n)$ is a decreasing function of $n$ with $\beta_1(\gdr{292})=\gdr{584.43}\ldots$. Therefore the bound (\ref{betabd}) is larger than the bound (\ref{simplebd}) for all $n\leq\gdr{292}$. We may therefore bound the quantity $|\mathbb{E}[h(F_r)]-\chi_{(r-1)}^2h|$ by (\ref{simplebd}) for $n\leq\gdr{292}$, and by (\ref{betabd}) for $n\geq\gdr{293}$. For fixed $r$, $\beta_2(n,r)$ and $\beta_3(n,r)$ are decreasing functions of $n$, and so evaluating the terms $\beta_1(n)$, $\beta_2(n,r)$ and $\beta_3(n,r)$ in the bound (\ref{betabd}) at $n=\gdr{293}$ yields a compact bound for $|\mathbb{E}[h(F_r)]-\chi_{(r-1)}^2h|$ that holds for all $n\geq1$. On rounding up numerical constants to the nearest integer this is our desired bound (\ref{thm2bound}). 

{\bf{Part V: Optimality of the rate.}} Let $h_t(x)=\cos(tx)$, and observe that $h_t\in C_b^{1,3}(\mathbb{R}^+)$. From the power series representation of the cosine function we can write $h_t(x)=\cos(tx)=1-t^2x^2/2+R_t(x)$, where $R_t(x)=\sum_{k=2}^\infty(-1)^k (tx)^{2k}/(2k)!$. Now, we let $Y_{r-1}\sim \chi_{(r-1)}^2$ and recall that $\mathbb{E}[Y_{r-1}^2]=r^2-1$, for $r\geq2$. Therefore using (\ref{frmoment2}) we have that
\begin{align*}|\mathbb{E}[h_t(F_r)]-\chi_{(r-1)}^2h_t|&=|(t^2/2)(\mathbb{E}[F_r^2]-\mathbb{E}[Y_{r-1}^2])+(\mathbb{E}[R_t(F_r)]-\mathbb{E}[R_t(Y_{r-1})])|\\
&=\bigg|\frac{t^2(r-1)}{n}-(\mathbb{E}[R_t(F_r)]-\mathbb{E}[R_t(Y_{r-1})])\bigg|.
\end{align*}
The quantity $\mathbb{E}[R_t(F_r)]-\mathbb{E}[R_t(Y_{r-1})]$ is a power series in $t$ which is linearly independent of $t^2$. Hence, the $r/n$ rate of the bound (\ref{thm2bound}) is optimal. \hfill $\Box$


\section{Further proofs}\label{appa}


\noindent{\bf{Proof of Proposition \ref{prop1.4}.}} We recall some bounds from \cite{gaunt normal}. Suppose $X_1,\ldots,X_n$ are i.i.d.\ random variables with $\mathbb{E}[X_1]=0$, $\mathbb{E}[X_1^2]=1$ and $\mathbb{E}[X_1^8]<\infty$. Let $W=n^{-1/2}\sum_{i=1}^nX_i$. Then inequalities (3.10) and (3.11) of \cite{gaunt normal} state that
\begin{equation}\label{aineq1}d_{\mathrm{W}}(\mathcal{L}(W^2),\chi_{(1)}^2)\leq\frac{48}{\sqrt{n}}\bigg[\mathbb{E}|X_1^3|+\sqrt{\frac{2}{\pi}}\mathbb{E}[X_1^4]+\frac{\mathbb{E}|X_1^5|}{\sqrt{n}}\bigg],
\end{equation}
and, for $h\in C_b^{1,2}(\mathbb{R}^+)$, if $\mathbb{E}[X_1^3]=0$,
\begin{align}\label{aineq2}|\mathbb{E}[h(W^2)]-\chi_{(1)}^2h|&\leq\frac{1}{n}\{\|h'\|+\|h''\|\}\bigg[22\mathbb{E}|W^3|+40\mathbb{E}|X_1^5|+\frac{43}{n}\mathbb{E}|X_1^7|\bigg].
\end{align}
Recall from Remark \ref{rem1.2} that we can write $F_2=2S_1^2$, where $S_1=n^{-1/2}\sum_{i=1}^nY_i$, where $Y_1,\ldots,Y_n$ are i.i.d.\ with $Y_1\sim\mathrm{Unif}\{-1/\sqrt{2},1/\sqrt{2}\}$. By rescaling,  $F_2=W^2$ with $W=n^{-1/2}\sum_{i=1}^nX_i$, where $X_1,\ldots,X_n$ are i.i.d.\ with $X_1\sim\mathrm{Unif}\{-1,1\}$. We have $\mathbb{E}[X_1]=\mathbb{E}[X_1^3]=0$ and $\mathbb{E}|X_1^m|=1$ for all $m\geq1$. We also have, by H\"older's inequality, that $\mathbb{E}|W^3|\leq\{\mathbb{E}[W^4]\}^{3/4}=\{3(n-1)(\mathbb{E}[X_1^2])^2/n+\mathbb{E}[X_1^4]/n\}^{3/4}\leq 3^{3/4}$. We have verified that the conditions under inequalities (\ref{aineq1}) and (\ref{aineq2}) hold are met, and plugging our moment estimates into these bounds yields the bounds (\ref{propbd1}) and (\ref{propbd2}), respectively. \hfill $\Box$

\vspace{3mm}

\noindent{\bf{Proof of Corollary \ref{cor1.3}.}} {First, we consider the case $r=2$. Recall from the proof of Proposition \ref{prop1.4} that we can write $F_2=W^2$, where $W=n^{-1/2}\sum_{i=1}^nX_i$, and $X_1,\ldots,X_n$ are i.i.d.\ with $X_1\sim\mathrm{Unif}\{-1,1\}$. Recall also that if $Z\sim N(0,1)$, then $Z^2\sim\chi_{(1)}^2$. Now, for any $z>0$,
\begin{align*}&|\mathbb{P}(W^2\leq z)-\mathbb{P}(Z^2\leq z)|=|\mathbb{P}(-\sqrt{z}\leq W\leq\sqrt{z})-\mathbb{P}(-\sqrt{z}\leq Z\leq\sqrt{z})|\\
&\leq|\mathbb{P}(W\leq\sqrt{z})-\mathbb{P}(Z\leq\sqrt{z})|+|\mathbb{P}(W\leq-\sqrt{z})-\mathbb{P}(Z\leq-\sqrt{z})|,
\end{align*}
and so
\begin{equation*}d_{\mathrm{K}}(\mathcal{L}(F_2),\chi_{(1)}^2)\leq 2d_{\mathrm{K}}(\mathcal{L}(W),\mathcal{L}(Z)).
\end{equation*}
Using the Berry-Esseen theorem with the best numerical constant currently available of $C=0.4748$ due to \cite{s11} we obtain the bound
\begin{align*}d_{\mathrm{K}}(\mathcal{L}(F_2),\chi_{(1)}^2)\leq 2\cdot\frac{0.4748\mathbb{E}|X_1^3|}{(\mathbb{E}[X_1^2])^{3/2}\sqrt{n}}=\frac{0.9496}{\sqrt{n}},
\end{align*}
where we used that $\mathbb{E}[X_1^2]=\mathbb{E}|X_1^3|=1$.

We now} consider the cases $r=3$ and $r\geq4$. Let $\alpha>0$, and for fixed $z>0$ define 
 \begin{equation*}
 g_z(x) = \left\{\begin{aligned} 1, & & \text{if $x \leq-1$,} & \\
   1-\tfrac{2}{3}(x+1)^3, & & \text{if $-1<x\leq-1/2$,} &
 \\ \tfrac{2}{3}x^3-x+\tfrac{1}{2}, & & \text{ if $-1/2<x\leq1/2$,} &
 \\ \tfrac{2}{3}(1-x)^3, & & \text{if $1/2<x\leq1$,} & \\
 0, & & \text{if $x>1$.} &
 \end{aligned} \right.
 \end{equation*}
Now let $h_{\alpha,z}(x)=g_z(1+2(x-z)/\alpha)$.
We observe that $h_{\alpha,z}''$ is Lipschitz, and simple calculations show that $\|h_{\alpha,z}'\|=2/\alpha$, $\|h_{\alpha,z}''\|=8/\alpha^2$ and $\|h_{\alpha,z}^{(3)}\|=32/\alpha^3$.  
Let $Y_{r-1}\sim\chi_{(r-1)}^2$. Then by applying  (\ref{thm2bound}) we obtain the bound
\begin{align}&\mathbb{P}(F_r\leq z)-\mathbb{P}(Y_{r-1}\leq z)\nonumber\\
&\leq \mathbb{E}[h_{\alpha,z}(F_r)]-\mathbb{E}[h_{\alpha,z}(Y_{r-1})]+\mathbb{E}[h_{\alpha,z}(Y_{r-1})]-\mathbb{P}(Y_{r-1}\leq z)\nonumber \\
&\leq \frac{r}{n}\bigg\{\gdr{585}\|h_{\alpha,z}'\|+\bigg(\gdr{2679}+\frac{431r}{n}\bigg)\|h_{\alpha,z}''\|+\bigg(\gdr{3905}+\frac{646r}{n}\bigg)\|h_{\alpha,z}^{(3)}\|\bigg\}+\mathbb{P}(z\leq Y_{r-1}\leq z+\alpha)\nonumber \\
\label{kol123}&=\frac{r}{n}\bigg\{\frac{\gdr{1170}}{\alpha}+\frac{\gdr{21\,432}}{\alpha^2}+\frac{3448r}{n\alpha^2}+\frac{\gdr{124\,960}}{\alpha^3}+\frac{20\,672r}{n\alpha^3}\bigg\}+\mathbb{P}(z\leq Y_{r-1}\leq z+\alpha).
\end{align}
The following inequality is a slight simplification of one given on p.\ 754  of \cite{gaunt chi square}:
\begin{equation}\label{ccbbcsb}\mathbb{P}(z\leq Y_{r-1}\leq z+\alpha)\leq \begin{cases} 
\alpha/2, & \: \mbox{if } r=3, \\
\displaystyle \alpha/\sqrt{\pi r}, & \:  \mbox {if } r\geq4. \end{cases}
\end{equation}
The simplification is that we used the inequality $r-3\geq r/4$ for $r\geq4$. We obtain upper bounds for $r=3$ and $r\geq 4$ by substituting inequality (\ref{ccbbcsb}) into (\ref{kol123}) and taking a suitable $\alpha$.  We choose $\alpha=\gdr{29.43}n^{-1/4}$ for $r=3$, and $\alpha=\gdr{21.69}r^{5/8}n^{-1/4}$ for $r\geq4$. The exponent $-1/4$ for $n$ optimises the rate of convergence of the Kolmogorov distance bound with respect to $n$. For the case $r\geq4$, the exponent $5/8$ for $r$ is chosen to allow for a Kolmogorov distance bound that converges to zero if $r^{1/2}/n\rightarrow0$; another choice of exponent would mean that the bound only converges to zero if $r^{\delta}/n\rightarrow0$, for some $\delta>1/2$. For both the cases $r=3$ and $r\geq4$, the numerical constants $\gdr{29.43}$ and $\gdr{21.69}$, respectively are chosen to minimise the numerical constants of the $n^{-1/4}$ terms in the final bounds. 
By  the same procedure one obtains a lower bound, which is the negative of the upper bound.  The proof of the bounds for the cases $r=3$ and $r\geq4$ is now complete. \hfill $\square$


\section*{Acknowledgements}
\gr{Firstly we acknowledge many helpful comments by anonymous referees. Moreover, we}
would like to thank Persi Diaconis for bringing this problem to our attention. During this research, RG was supported \gdr{in part} by 
EPSRC grant EP/K032402/1 and 
 a Dame Kathleen Ollerenshaw Research
Fellowship.  GR  has been  funded in part by EPSRC grants EP/K032402/1, EP/T018445/1 and EP/R018472/1.


\normalsize

\appendix


\section{Supplementary Material}

\subsection{Proof of the lemmas from Section \ref{sec3}}  \label{appproofs}


{\bf Proof of Lemma \ref{robcov}.}
For fixed $j$, $\rho_1(j),\ldots,\rho_n(j)$ are i.i.d.\ $\mathrm{Unif}\{1,\ldots,r\}$ random variables, and so
\[\sigma_{jj}=\mathrm{Var}(S_j)=\frac{12}{r(r+1)n}\sum_{i=1}^n\mathrm{Var}(\rho_i(j))=\frac{12}{r(r+1)n}\times n\times\frac{r^2-1}{12}=\frac{r-1}{r}.\]
Suppose now that $j\not=k$.  Since $\sum_{j=1}^rS_j=0$ and the $S_j$ are identically distributed, we have
\[0=\mathbb{E}\bigg[S_j\sum_{l=1}^rS_l\bigg]=\mathbb{E}[S_j^2]+\sum_{l\not= j}\mathbb{E}[S_jS_l]=\mathbb{E}[S_j^2]+(r-1)\mathbb{E}[S_jS_k].\]
On rearranging, and using that $\mathbb{E}[S_j^2]=(r-1)/r$, we have that $\mathbb{E}[S_jS_k]=-1/r$ for $j\not=k$.  As $\mathbb{E}[S_j]=0$, it follows that $\sigma_{jk}=\mathrm{Cov}(S_j,S_k)=\mathbb{E}[S_jS_k]=-1/r$.
\hfill $\Box$

\bigskip
{\noindent{\bf Proof of Lemma \ref{lemexpformulas}.}}
The equalities
  $\mathbb{E}[\rho_m(l)]=\mathbb{E}[\rho_m(l)^3]=\mathbb{E}[\rho_m(l)^2\rho_m(j)]=0$ and $\mathbb{E}[\rho_m(l)\rho_m(j)\rho_m(s)]=0$ follow by symmetry. The formula for $\mathbb{E}[\rho_m(l)^2]$ follows from the standard variance formula for the uniform distribution. Using that $\sum_{k=1}^r\rho_m(k)=0$ gives that, for $l\not=j$,
\begin{equation}\label{sumeqn}0=\mathbb{E}\bigg[\rho_m(l)\sum_{k=1}^r\rho_m(k)\bigg]=\mathbb{E}[\rho_m^2(l)]+(r-1)\mathbb{E}[\rho_m(l)\rho_m(j)],
\end{equation}
whence on using 
 $\mathbb{E}[\rho_m(l)^2]=(r^2-1)/12$ and rearranging gives the desired formula for $\mathbb{E}[\rho_m(l)\rho_m(j)]$.
The bound for $\mathbb{E}[\rho_m(l)^4]$ 
follows from 
 $\mathbb{E}[\rho_m(l)^4]=(r^2-1)(3r^2-7)/240$. Arguing as 
  in (\ref{sumeqn}), 
   for $r\geq2$, 
\begin{align}|\mathbb{E}[\rho_m(l)^3\rho_m(j)]|&=\bigg|-\frac{\mathbb{E}[\rho_m(l)^4]}{r-1}\bigg|=\frac{(r+1)(3r^2-7)}{240}\leq\frac{r^3}{80},\nonumber\\
\mathbb{E}[\rho_m(l)^2\rho_m(j)^2]&=\frac{1}{r-1}\Big(r\big(\mathbb{E}[\rho_m(l)^2]\big)^2-\mathbb{E}[\rho_m(l)^4]\Big)\nonumber\\
\label{rhoyyy}&=\frac{(r+1)(5r^3-9r^2-5r+21)}{720}\leq\frac{r^4}{144}.
\end{align}
For $r\geq3$ and $r\geq4$, respectively, we have
\begin{align*}|\mathbb{E}[\rho_m(l)^2\rho_m(j)\rho_m(s)]|&=\bigg|-\frac{1}{r-2}\mathbb{E}[\rho_m(l)^3\rho_m(j)+\rho_m(l)^2\rho_m(j)^2]\bigg|\\
&=\bigg|-\frac{(r-3)(r+1)(5r+7)}{720}\bigg|\leq\frac{r^3}{144}. \\
|\mathbb{E}[\rho_m(l)\rho_m(j)\rho_m(s)\rho_m(t)]|&=\bigg|-\frac{3}{r-3}\mathbb{E}[\rho_m(l)^2\rho_m(j)\rho_m(s)]\bigg|=\frac{(r+1)(5r+7)}{240}.
\end{align*}
For $r\geq2$, 
independence as well as  $\mathbb{E}[\rho_1(j)]=0$ and the formulas for $\mathbb{E}[\rho_1(j)^2]$ and $\mathbb{E}[\rho_1(j)^4]$ give
\begin{align}\mathbb{E}[S_j^4]&=\frac{144}{r^2(r+1)^2}\bigg\{\frac{3(n-1)}{n}(\mathbb{E}[\rho_1(j)^2])^2+\frac{1}{n}\mathbb{E}[\rho_1(j)^4]\bigg\}\nonumber\\
\label{sj22}&=\frac{3(r-1)((5n-2)r^2-5n-2)}{5nr^2(r+1)},\\
&\leq \frac{3(r-1)(5n-2)(r^2-1)}{5nr^2(r+1)}=\frac{3(r-1)^2}{r^2}\bigg(1-\frac{2}{5n}\bigg)\leq 3-\frac{6}{5n},\nonumber
\end{align}
as required.
\hfill $\Box$

\bigskip
{\noindent{\bf Proof of Lemma \ref{lemmatm}.}} 
We begin by obtaining a useful representation of $T_m$:
 \begin{align*} 
   T_m & = \sum_{l=1}^r S_l \rho_m(l)  
      =   \frac{\sqrt{12}}{\sqrt{r(r+1)n}} \bigg( \sum_{l=1}^r \rho_m(l)^2 + \sum_{k \ne m} \sum_{l=1}^r \rho_m(l) \rho_k(l) \bigg)\\
      &=   \frac{\sqrt{12}}{\sqrt{r(r+1)n}} \bigg( \frac{r(r^2-1)}{12} + \sum_{k: k \ne m} \sum_{l=1}^r \rho_m(l) \rho_k(l) \bigg),
   \end{align*} 
since $\sum_{l=1}^r \rho_m(l)^2=\sum_{l=1}^r\big(l-(r+1)/2\big)^2=r(r^2-1)/12$.  
With this representation, it is easy to see that by the independence of $\rho_m(l)$ and $\rho_k(l)$, for $ k \ne m$, 
$$ \mathbb{E} [T_m] = \frac{\sqrt{12}}{\sqrt{r(r+1)n}} \frac{r(r^2-1)}{12}  =  \frac{ (r-1)\sqrt{r(r+1)} }{ \sqrt{12} \sqrt{n}} .$$ 
Moreover, again using independence, as well as formulas from Lemma \ref{lemexpformulas}, 
  \begin{align*}
\mathbb{E}[T_m^2]&=\big(\mathbb{E}[T_m]\big)^2+ \mathrm{Var}(  T_m ) =\big(\mathbb{E}[T_m]\big)^2+\frac{{12}}{{r(r+1)n}}  \mathbb{E} \left[  \bigg( \sum_{k:k \ne m} \sum_{l=1}^r \rho_m(l) \rho_k(l) \bigg)^2 \right] \\
&=\big(\mathbb{E}[T_m]\big)^2+ \frac{{12}}{{ r(r+1)n}} \sum_{k:k \ne m} \sum_{l=1}^r    \mathbb{E}  [   \rho_m(l)^2  ] \mathbb{E}  [ \rho_k(l) ^2 ] \\
&\quad+ \frac{{12}}{{ r(r+1)n}} \sum_{k:k \ne m} \sum_{l=1}^r \sum_{j: j \ne l}   \mathbb{E}  \left[   \rho_m(l)  \rho_m(j)  \right] \mathbb{E}  \left[ \rho_k(l) \rho_k(j)  \right]\\
& =\frac{ (r-1)^2{r(r+1)} }{ {12} {n}} +
\frac{{12}}{{ r(r+1)n}}  (n-1) r\left(  \frac{(r^2-1)^2}{144}  +(r-1) \frac{(r+1)^2}{144} \right)\\
&= \frac{r  (r^2-1)}{{ 12 }}  \left( 1 + \frac{r-2}{n} \right)<\frac{r^3}{12}\bigg(1+\frac{r}{n}\bigg).
\end{align*}  
 
 To bound $ \mathbb{E}[  T_m^4]$ we abbreviate $T_m = a + b \sum_{k: k \ne m}  \beta_k$ with 
\begin{align*} a=
 \frac{ (r-1)\sqrt{r(r+1)} }{ \sqrt{12} \sqrt{n}}, \quad
b=  \frac{\sqrt{12}}{\sqrt{r(r+1)n}}, \quad
\beta_k=   \sum_{l=1}^r \rho_m(l) \rho_k(l).
 \end{align*}
 Then as $k\ne m$, we have that for $s,k,t$ distinct, 
$
 \mathbb{E}[ \beta_k] =
  \mathbb{E}[ \beta_k \beta_s ] =
  \mathbb{E}[ \beta_k^3] =
  \mathbb{E}[ \beta_k^2 \beta_s ]=
  \mathbb{E}[ \beta_k \beta_s  \beta_t] = 0
$,
   so that 
 \begin{align*}
  \mathbb{E}[  T_m^4]
  &= a^4 + 6 a^2 b^2  \sum_{k: k \ne m}   \mathbb{E} [ \beta_k^2] 
  + 6 b^4 \bigg( \sum_{k: k \ne m}   \mathbb{E} [ \beta_k^4] + 3 \sum_{k: k \ne m} \sum_{t: t \ne m, k}   \mathbb{E} [ \beta_k^2]     \mathbb{E} [ \beta_t^2]  \bigg) .
    \end{align*}
 We note that  
 \begin{align}\label{cont1} a^4 =  \frac{ (r-1)^4 r^2(r+1)^2 }{ 144  n^2}
  \le r^6  \left( \frac{  r}{ 12 n} \right)^2.
      \end{align}
 Next 
  we calculate 
 \begin{align*}
 \mathbb{E}[ \beta_k^2] &= 
 \sum_{l=1}^r \sum_{j=1}^r \mathbb{E}[   \rho_m(l) \rho_k(l)   \rho_m(j) \rho_k(j) ] = \sum_{l=1}^r  ( \mathbb{E}[   \rho_m(l)^2] )^2 +  \sum_{l=1}^r \sum_{j\ne l } ( \mathbb{E}[   \rho_m(l) \rho_k(l)] )^2\\
 &= r \left( \frac{r^2-1}{12} \right)^2 + r(r-1) \left( \frac{r+1}{12}  \right)^2  = \frac{r^2 (r+1)^2 (r-1)}{144}. 
   \end{align*}
 Hence the second contribution to $\mathbb{E}[  T_m^4]$ is 
 \begin{align}\label{cont2}
   6 a^2 b^2  \sum_{k: k \ne m}   \mathbb{E} [ \beta_k^2]
  &= 6 \frac{ (r-1)^2{r(r+1)} }{{12}{n}}  \frac{{12}}{{r(r+1)n}} (n-1) \frac{r^2 (r+1)^2 (r-1)}{144}
\nonumber  \\
  &\le   \frac{ r^2 (r+1)^2(r-1)^3 }{24n} \le r^6 \frac{ r}{24n}\leq \frac{r^6}{48}\bigg(1+\frac{r^2}{n^2}\bigg).
   \end{align}
 Similarly, the contribution 
  \begin{align}\label{cont3}
  18 b^4 \sum_{k: k \ne m} \sum_{t: t \ne m, k}   \mathbb{E} [ \beta_k^2]     \mathbb{E} [ \beta_t^2] 
= 18 \frac{144(n-1)(n-2)}{r^2(r+1)^2n^2}  \left( \frac{r^2 (r+1)^2 (r-1)}{144} \right)^2  \le \frac{r^6}{ 8}. 
   \end{align} 
 Next, again by independence, followed by an application of the bounds of Lemma \ref{lemexpformulas} we obtain
  \begin{align*}
 \mathbb{E}[ \beta_k^4] &= 
 \sum_{l, j, s, t=1}^r  \mathbb{E}[   \rho_m(l)  \rho_m(j)
   \rho_m(s)   \rho_m(t) \rho_k(l)  \rho_k(j)
   \rho_k(s)   \rho_k(t)   ]\\
   &=  \sum_{l, j, s, t=1}^r \left( \mathbb{E}[   \rho_m(l)  \rho_m(j)
   \rho_m(s)   \rho_m(t) ] \right)^2
 \\
   &= \sum_{l=1}^r    \left(  \mathbb{E}[   \rho_m(l)^4 ]  \right)^2 +  \sum_{l=1}^r \sum_{j\ne l}\Big(3 \big( \mathbb{E}[ \rho_m(l)^2   \rho_m(j)^2
 ]  \big)^2 + 4 \big( \mathbb{E}[   \rho_m(l)^3   \rho_m(j) 
 ]  \big)^2\Big) \\
 &   \quad+ 6   \sum_{l=1}^r \sum_{j\ne l}   \sum_{s\ne l, j} \left(  \mathbb{E}[   \rho_m(l)^2  \rho_m(j)
   \rho_m(s) ]  \right)^2 \\
   &\quad+ \sum_{l=1}^r \sum_{j\ne l}   \sum_{s\ne j,l}    \sum_{t\ne l,j.s} \left(  \mathbb{E}[   \rho_m(l)  \rho_m(j)
   \rho_m(s)   \rho_m(t) ]  \right)^2 \\
   &\leq r\bigg(\frac{r^4}{80}\bigg)^2+3r(r-1)\bigg(\frac{r^4}{144}\bigg)^2+4r(r-1)\bigg(\frac{r^3}{80}\bigg)^2+6 r(r-1)(r-2)\bigg(\frac{r^3}{144}\bigg)^2\\
   &\quad+r(r-1)(r-2)(r-3)\bigg(\frac{(r+1)(5r+7)}{240}\bigg)^2\leq\frac{79r^{10}}{345600},
  \end{align*} 
where we used that $r\geq2$ in obtaining the final inequality.  The argument for obtaining the constants in this expansion can be found in Appendix \ref{app2two}.

The contribution of the following term to $\mathbb{E}[  T_m^4]$ is  therefore
  \begin{align}\label{cont4}
   6 b^4 \sum_{k: k \ne m}   \mathbb{E} [ \beta_k^4] 
   &\le \frac{6 \times 144 }{r^2 (r+1)^2 n^2 } (n-1) \frac{79r^{10}}{345600}  <  \frac{r^6}{5n}.
  \end{align} 
Summing up (\ref{cont1}), (\ref{cont2}), (\ref{cont3}) and (\ref{cont4}) now yields the bound (\ref{crt}). 
\hfill $\Box$

\bigskip
{\noindent{\bf Proof of Lemma \ref{frmeanvar}.}}
The formula for $\mathbb{E}[F_r]$ is immediate from (\ref{miltonfa}) and the formula $\mathbb{E}[S_j^2]=\mathrm{Var}(S_j)=(r-1)/r$ (see Lemma \ref{robcov}). Let us now prove the formula for $\mathbb{E}[F_r^2]$. We have that
\begin{align}\label{efr2}\mathbb{E}[F_r^2]=\mathbb{E}\bigg[\bigg(\sum_{j=1}^rS_j^2\bigg)^2\bigg]=\sum_{j=1}^r\mathbb{E}[S_j^4]+\sum_{j=1}^r\sum_{k\not=j}\mathbb{E}[S_j^2S_k^2].
\end{align}
A formula for $\mathbb{E}[S_j^4]$ is given in (\ref{sj22}). We now calculate $\mathbb{E}[S_j^2S_k^2]$, $j\not=k$. Using independence and the expectation formulas given in (\ref{scov}) and (\ref{rhoyyy}) gives that, for $j\not=k$,
\begin{align*}\mathbb{E}[S_j^2S_k^2]&=\frac{144}{r^2(r+1)^2n^2}\bigg\{\sum_{i=1}^n\mathbb{E}[\rho_i(j)^2\rho_i(k)^2]+\sum_{i=1}^n\sum_{l\not=i}\Big(\mathbb{E}[\rho_i(j)^2]\mathbb{E}[\rho_l(k)^2]\\
&\quad+2\mathbb{E}[\rho_i(j)\rho_i(k)]\mathbb{E}[\rho_l(k)\rho_l(k)]\Big)\bigg\}\\
&=\frac{144}{r^2(r+1)^2}\bigg\{\frac{(r+1)(5r^3-9r^2-5r+21 )}{720n}+\bigg(1-\frac{1}{n}\bigg)\bigg[\frac{(r^2-1)^2}{144}+2\frac{(r+1)^2}{144}\bigg]\bigg\}\\
&=\frac{5n(r^3-r^2+r+3)-4r^2-10r+6}{5nr^2(r+1)}.
\end{align*}

Plugging this formula and (\ref{sj22}) into (\ref{efr2}) yields
\begin{align*}\mathbb{E}[F_r^2]&=\frac{3(r-1)((5n-2)r^2-5n-2)}{5nr^2(r+1)}+\frac{5n(r-1)(r^3-r^2+r+3)-4r^2-10r+6}{5nr(r+1)}\\
&=r^2-1-\frac{2(r-1)}{n}.
\end{align*}
Lastly, we deduce (\ref{frvar}) by the elementary formula $\mathrm{Var}(F_r)=\mathbb{E}[F_r^2]-(\mathbb{E}[F_r])^2$.
\hfill $\Box$

\bigskip
\noindent{\bf{Proof of Lemma \ref{suplem1}.}}
A direct calculation gives that, for $r\geq2$,
\begin{align*}\mathbb{E}\bigg[\bigg( \frac{ (r^2-1)}{4} \rho_m(j)  +  \rho_m(j)^3\bigg)^2\rho_m(j)^2\bigg]&=\frac{(r^2-1)(47r^6-322r^4+875r^2-936)}{20160}\\
&\leq\frac{47r^8}{20160}\leq 0.00234r^8.
\end{align*}
By the Cauchy-Schwarz inequality and direct calculations we have that, for $r\geq2$,
\begin{align*}\mathbb{E}\bigg[\bigg( \frac{ (r^2-1)}{4} \rho_m(j)  +  \rho_m(j)^3\bigg)^2\rho_m(k)^2\bigg]&\leq\bigg(\mathbb{E}\bigg[\bigg( \frac{ (r^2-1)}{4} \rho_m(j)  +  \rho_m(j)^3\bigg)^4\bigg]\mathbb{E}[\rho_m(k)^4]\bigg)^{1/2}\\
&\leq\bigg(\frac{1763r^{12}}{3\,843\,840}\times\frac{r^4}{80}\bigg)^{1/2}\leq 0.00240r^8,
\end{align*}
where $(1763/3\,843\,840)r^{12}$ is an upper bound for a polynomial in $r$ of degree 12 which we do not reproduce. \hfill $\Box$

\bigskip

\begin{lemma}Suppose $r\geq2$. Then, for $j=1,\ldots,r$,
\begin{equation}\label{sj6bd}\mathbb{E}[S_j^6]\leq15.
\end{equation}
\end{lemma}

\begin{proof}Using independence and the formulas $\mathbb{E}[\rho_1(j)^2]=(r^2-1)/12$,  $\mathbb{E}[\rho_1(j)^3]=0$, $\mathbb{E}[\rho_1(j)^4]=(r^2-1)(3r^2-7)/240$ and $\mathbb{E}[\rho_1(j)^6]=(r^2-1)(3r^4-18r^2+31)/1344$ we obtain that
\begin{align*}\mathbb{E}[S_j^6]&=\frac{1728}{r^3(r+1)^3}\bigg\{\frac{15(n-1)(n-2)}{n^2}(\mathbb{E}[\rho_1(j)^2])^3+\frac{10(n-1)}{n^2}(\mathbb{E}[\rho_1(j)^3])^2\\
&\quad+\frac{15(n-1)}{n^2}\mathbb{E}[\rho_1(j)^4]\mathbb{E}[\rho_1(j)^2]+\frac{1}{n^2}\mathbb{E}[\rho_1(j)^6]\bigg\}\\
&=\frac{3(r-1)(35n^2(r^2-1)^2-42n(r^4-1)+16(r^4+r^2+1))}{7r^3(r+1)^2n^2}\leq15.
\end{align*}
as required.
\end{proof}

\bigskip 
\noindent{\bf{Proof of Lemma \ref{suplem1}.}} Since $\sum_{l=1}^r(l - (r+1)/2 )^q=0$ for odd $q$, we have that
\begin{align*}\sum_{l=1}^r ( \rho_m(k) - \rho_m(l) )^4&=\sum_{l=1}^r \bigg( \rho_m(k)-  \Big(l - \frac{r+1}{2}\Big)  \bigg)^4\\
&=\sum_{l=1}^r\bigg(l - \frac{r+1}{2}\bigg)^4+6\rho_m(k)^2\sum_{l=1}^r\bigg(l - \frac{r+1}{2}\bigg)^2+r\rho_m(k)^4\\
&=\frac{r(3r^4-10r^2+7)}{240}+\frac{r(r^2-1)}{2}\rho_m(k)^2+r\rho_m(k)^4.
\end{align*}
Similarly,
\begin{align*}\sum_{l=1}^r ( \rho_m(k) - \rho_m(l) )^6&=\frac{r(3r^6-21r^4+49r^2-31)}{1344}+\frac{r(3r^4-10r^2+7)}{16}\rho_m(k)^2\\
&\quad+\frac{5r(r^2-1)}{4}\rho_m(k)^4+r\rho_m(k)^6.
\end{align*}
Using these formulas, the Cauchy-Schwarz and H\"older inequalities and the inequalities $\mathbb{E}[S_k^2]\leq1$, $\mathbb{E}[S_k^4]\leq3$ and $\mathbb{E}[S_k^6]\leq15$ now yields the bounds
\begin{align*}&\sum_{l=1}^r\mathbb{E}[S_k^2(\rho_m(l)-\rho_m(k)^4]\\
&\leq\frac{r(3r^4-10r^2+7)}{240}\mathbb{E}[S_k^2]+\frac{r(r^2-1)}{2}\big(\mathbb{E}[S_k^4]\mathbb{E}[\rho_m(k)^4]\big)^{1/2}+r\big(\mathbb{E}[S_k^4]\mathbb{E}[\rho_m(k)^8]\big)^{1/2}\\
&\leq\frac{r(3r^4-10r^2+7)}{240}+\frac{r(r^2-1)}{2}\bigg(3\times\frac{r^4}{80}\bigg)^{1/2}+r\bigg(3\times\frac{r^8}{2304}\bigg)^{1/2}\\
&\leq\frac{3+5\sqrt{3}+6\sqrt{15}}{250}r^5<0.1455r^5,
\end{align*}
and
\begin{align*}&\sum_{l=1}^r\mathbb{E}[S_k^4(\rho_m(l)-\rho_m(k)^4]\\
&\leq\frac{r(3r^4-10r^2+7)}{240}\mathbb{E}[S_k^4]+\frac{r(r^2-1)}{2}\big(\mathbb{E}[S_k^6]\big)^{2/3}\big(\mathbb{E}[\rho_m(k)^6]\big)^{1/3}\\
&\quad+r\big(\mathbb{E}[S_k^6]\big)^{2/3}\big(\mathbb{E}[\rho_m(k)^{12}]\big)^{1/3}\\
&\leq\frac{r(3r^4-10r^2+7)}{80}+\frac{r(r^2-1)}{2}\bigg(15^2\times\frac{r^6}{448}\bigg)^{1/3}+r\bigg(15^2\times\frac{r^{12}}{53248}\bigg)^{1/3}<0.6717r^5,
\end{align*}
and finally
\begin{align*}&\sum_{l=1}^r\mathbb{E}[S_k^2(\rho_m(l)-\rho_m(k)^6]\\
&\leq\frac{r(3r^6-21r^4+49r^2-31)}{1344}\mathbb{E}[S_k^2]+\frac{r(3r^4-10r^2+7)}{16}\big(\mathbb{E}[S_k^4]\mathbb{E}[\rho_m(k)^4]\big)^{1/2}\\
&\quad+\frac{5r(r^2-1)}{4}\big(\mathbb{E}[S_k^4]\mathbb{E}[\rho_m(k)^8]\big)^{1/2}+r\big(\mathbb{E}[S_k^4]\mathbb{E}[\rho_m(k)^{12}]\big)^{1/2}\\
&\leq\frac{r(3r^6-21r^4+49r^2-31)}{1344}+\frac{r(3r^4-10r^2+7)}{16}\bigg(3\times\frac{r^4}{80}\bigg)^{1/2}\\
&\quad+\frac{5r(r^2-1)}{4}\bigg(3\times\frac{r^8}{2304}\bigg)^{1/2}+r\bigg(3\times\frac{r^{12}}{53248}\bigg)^{1/2}<0.09116r^7.
\end{align*}
\hfill $\Box$

\vspace{3mm}

\noindent{\bf{Proof of Lemma \ref{suplem2}.}} A direct calculation gives that
\begin{align*}\mathbb{E}\big[((r^2-1)-12\rho_m(j)^2)^2\rho_m(j)^4\big]=\frac{1}{420}(r^2-1)(r^2-4)(9r^4-118r^2+445)\leq\frac{3r^8}{140}.
\end{align*}
By the Cauchy-Schwarz inequality and direct calculations we obtain the bound
\begin{align*}\mathbb{E}\big[((r^2-1)-12\rho_m(j)^2)^2\rho_m(q)^4\big]&\leq\Big(\mathbb{E}\big[((r^2-1)-12\rho_m(j)^2)^4\big]\mathbb{E}[\rho_m(q)^8]\Big)^{1/2}\\
&\leq\bigg(\frac{48}{35}(r^2-1)(r^2-4)(r^4-17r^2+100)\times\frac{r^8}{2304}\bigg)^{1/2}\\
&\leq\sqrt{\frac{48r^8}{35}\times\frac{r^8}{2304}}=\frac{\sqrt{105}}{420}r^8< 0.02440r^8.
\end{align*}
Now, with the basic inequality $ab\leq(a^2+b^2)/2$ we deduce the bound
\begin{align*}\mathbb{E}\big[((r^2-1)-12\rho_m(j)^2)^2\rho_m(j)^2\rho_m(q)^2\big]&\leq \frac{1}{2}\Big(\mathbb{E}\big[((r^2-1)-12\rho_m(j)^2)^2\rho_m(j)^4\big]\\
&\quad+\mathbb{E}\big[((r^2-1)-12\rho_m(j)^2)^2\rho_m(j)^4\big]\Big)\\
&\leq\frac{1}{2}\bigg(\frac{3r^8}{140}+0.02440r^8\bigg)<0.02292r^8,
\end{align*}
Finally, arguing as before we obtain
\begin{align*}\mathbb{E}\big[(    (r^2-1) - 12  \rho_m(j)^2 \big)^2   \rho_m(q)^4 \rho_m(t)^4\big]&\leq\Big(\mathbb{E}\big[((r^2-1)-12\rho_m(j)^2)^4\big]\mathbb{E}[\rho_m(q)^8\rho_m(t)^8]\Big)^{1/2}\\
&\leq\bigg(\frac{48r^8}{35}\sqrt{\mathbb{E}[\rho_m(q)^{16}]\mathbb{E}[\rho_m(t)^{16}]}\bigg)^{1/2}\\
&\leq\bigg(\frac{48r^8}{35}\times\frac{r^{16}}{1\,114\,112}\bigg)^{1/2}<0.00111r^{12},
\end{align*}
as required. \hfill $\Box$

\bigskip 
\noindent{\bf{Proof of Lemma \ref{suplem3}.}} The bound (\ref{suplem3for}) clearly holds if $k=l$, so we now suppose that $k\not=l$. Using formulas from Lemma \ref{lemexpformulas} we have that
\begin{align*}\mathbb{E}[\rho_m(k)-\rho_m(l))^2]=2\mathbb{E}[\rho_m(k)^2]-2\mathbb{E}[\rho_m(k)\rho_m(l)]=\frac{r^2-1}{6}+\frac{r+1}{6}=\frac{r(r+1)}{6}.
\end{align*}
Also, using formulas for the expectations $\mathbb{E}[\rho_m(k)^4]$, $\mathbb{E}[\rho_m(k)^3\rho_m(l)]$ and $\mathbb{E}[\rho_m(k)^2\rho_m(l)^2]$ given in the proof of Lemma \ref{lemexpformulas} we have that
\begin{align*}\mathbb{E}[\rho_m(k)-\rho_m(l))^4]&=2\mathbb{E}[\rho_m(k)^4]-8\mathbb{E}[\rho_m(k)^3\rho_m(l)]+6\mathbb{E}[\rho_m(k)^2\rho_m(l)^2]\\
&=\frac{(r^2-1)(3r^2-7)}{120}+\frac{(r+1)(3r^2-7)}{30}\\
&\quad+\frac{(r+1)(5r^3-9r^2-5r+21)}{720}\\
&=\frac{r(r+1)(2r^2-3)}{30}.
\end{align*}
With these formulas we deduce that
\begin{align*}\mathbb{E}\bigg[\bigg(\frac{6}{r(r+1)}(\rho_m(k)-\rho_m(l))^2-1\bigg)^2\bigg]&=\frac{36}{r^2(r+1)^2}\frac{r(r+1)(2r^2-3)}{30}\\
&\quad-\frac{12}{r(r+1)}\frac{r(r+1)}{6}-1\\
&=\frac{(r-2)(7r+9)}{5r(r+1)}\leq\frac{7}{5}.
\end{align*}
\hfill $\Box$


\subsection{Converting a general sum over up to 4 indices into a sum over distinct indices}\label{app2two} 
For a function $f$ which is invariant under permutation of its arguments, 
\begin{align*}
\sum_{l, j=1}^r f(l,j) &= \sum_{l=1}^r f(l,l) + \sum_{l, j\ne l}^r f(l,j), 
\end{align*}
\begin{align*}
\sum_{l, j, s=1}^r f(l,j,s) &= \sum_{l, j =1}^r f(l,j,j) + \sum_{l, j, s\ne j}^r f(l,j,s)\\
&= \sum_{l, j =1}^r f(l,j,j) +  \sum_{j,  s\ne j}^r f(j,j,s)+ \sum_{l, j \ne l,  s\ne j}^r f(l,j,s)\\
&= \sum_{l, j =1}^r f(l,j,j) +  \sum_{j,  s\ne j}^r f(j,j,s)+ \sum_{l, j \ne l}^r f(l,j,l) +  \sum_{l, j \ne l,  s\ne j, l}^r f(l,j,s)
\\
&= \sum_{j =1}^r f(j,j,j) +\sum_{l, j \ne l}^r f(l,j,j) + \sum_{j,  s\ne j}^r f(j,j,s)+ \sum_{l, j \ne l}^r f(l,j,l)\\
&\quad +  \sum_{l, j \ne l,  s\ne j, l}^r f(l,j,s)
\\
&= \sum_{j =1}^r f(j,j,j) +3 \sum_{l, j \ne l}^r f(l,j,j) +  \sum_{l, j \ne l,  s\ne j, l}^r f(l,j,s)
\end{align*}
and 
{{
\begin{align*}
 \lefteqn{\sum_{l, j, s, t=1}^r f(l,j,s,t)}\\
 &=  \sum_{l, j, s=1}^r f(l,j,s,s)   +  \sum_{l, j, s=1}^r  \sum_{t \ne s} f(l,j,s,t) \\
 &=  \sum_{l, j=1}^rf(l,j,j,j)  + \sum_{l, j=1}^r  \sum_{s \ne j} f(l,j,s,s)  
 +  \sum_{l, s=1}^r  \sum_{t \ne s} f(l,l,s,t)  +  \sum_{l, s=1}^r \sum_{j \ne l}  \sum_{t \ne s} f(l,j,s,t)  \\
 &=  \sum_{j=1}^rf(j,j,j,j)  + \sum_{l, j\ne l }^rf(l,j,j,j)  +\sum_{l=1}^r  \sum_{s \ne j} f(l,l,s,s)  + \sum_{l=1}^r  \sum_{j \ne l } \sum_{s \ne j} f(l,j,s,s)\\
 &
\quad +  \sum_{l=1}^r  \sum_{t \ne s} f(l,l,l,t)  +  \sum_{l=1}^r \sum_{s \ne l} \sum_{t \ne s} f(l,l,s,t) +  \sum_{l=1}^r \sum_{j \ne l}  \sum_{t \ne l} f(l,j,l,t)  \\
 &\quad +  \sum_{l=1}^r \sum_{s \ne l}  \sum_{j \ne l}  \sum_{t \ne s} f(l,j,s,t) 
  \\
   &=  \sum_{j=1}^rf(j,j,j,j)  + 2  \sum_{l, j\ne l }^rf(l,j,j,j)  +\sum_{l=1}^r  \sum_{s \ne j} f(l,l,s,s)  +
   2  \sum_{l=1}^r  \sum_{j \ne l } \sum_{s \ne j} f(l,j,s,s)  \\
 &  \quad+  \sum_{l=1}^r \sum_{j \ne l}  \sum_{t \ne l} f(l,j,l,t)  +  \sum_{l=1}^r \sum_{s \ne l}  \sum_{j \ne l}  \sum_{t \ne s} f(l,j,s,t) 
  \\
   &=  \sum_{j=1}^rf(j,j,j,j)  + 2  \sum_{l, j\ne l }^rf(l,j,j,j)  \quad+\sum_{l=1}^r  \sum_{s \ne j} f(l,l,s,s)  +
   2  \sum_{l=1}^r  \sum_{j \ne l }  f(l,j,l,l)  
    \end{align*}
 \begin{align*}
 &\quad +    2  \sum_{l=1}^r  \sum_{j \ne l } \sum_{s \ne j, l} f(l,j,s,s)  \\
 &\quad +  \sum_{l=1}^r \sum_{j \ne l}  \sum_{t \ne l} f(l,j,l,t) +   \sum_{l=1}^r \sum_{s \ne l}  \sum_{j \ne l}   f(l,j,s,l) +  \sum_{l=1}^r \sum_{s \ne l}  \sum_{j \ne l}  \sum_{t \ne s, l} f(l,j,s,t)  \\
   &=  \sum_{j=1}^rf(j,j,j,j)  + 2  \sum_{l, j\ne l }^rf(l,j,j,j)  +\sum_{l=1}^r  \sum_{s \ne j} f(l,l,s,s)  +
   2  \sum_{l=1}^r  \sum_{j \ne l }  f(l,j,l,l)  \\
 & \quad+    2  \sum_{l=1}^r  \sum_{j \ne l } \sum_{s \ne j, l} f(l,j,s,s)   +  2 \sum_{l=1}^r \sum_{j \ne l}  \sum_{t \ne l} f(l,j,l,t) +  \sum_{l=1}^r \sum_{s \ne l}  \sum_{j \ne l}  \sum_{t \ne s, l} f(l,j,s,t)  \\
  &=  \sum_{j=1}^rf(j,j,j,j)  + 4  \sum_{l, j\ne l }^rf(l,j,j,j)  +\sum_{l=1}^r  \sum_{s \ne j} f(l,l,s,s)  \\
 &\quad +    2  \sum_{l=1}^r  \sum_{j \ne l } \sum_{s \ne j, l} f(l,j,s,s)  +  2 \sum_{l=1}^r \sum_{j \ne l}  f(l,j,l,j) 
 +  2 \sum_{l=1}^r \sum_{j \ne l}  \sum_{t \ne j, l} f(l,j,l,t)\\
 & \quad +  \sum_{l=1}^r \sum_{s \ne l}  \sum_{j \ne l}  f(l,j,s,j) 
 +  \sum_{l=1}^r \sum_{s \ne l}  \sum_{j \ne l}  \sum_{t \ne s,j,l} f(l,j,s,t)  \\
 &=  \sum_{j=1}^rf(j,j,j,j)  + 4  \sum_{l, j\ne l }^rf(l,j,j,j)  +\sum_{l=1}^r  \sum_{s \ne j} f(l,l,s,s)  \\
 & \quad+    2  \sum_{l=1}^r  \sum_{j \ne l } \sum_{s \ne j, l} f(l,j,s,s)  +  2 \sum_{l=1}^r \sum_{j \ne l}  f(l,j,l,j) 
 +  2 \sum_{l=1}^r \sum_{j \ne l}  \sum_{t \ne j, l} f(l,j,l,t)\\
 & \quad  +  \sum_{l=1}^r \sum_{s \ne l}    \sum_{t \ne s,j,l} f(l,s,s,t)   +  \sum_{l=1}^r \sum_{s \ne l}  \sum_{j \ne l, s }  \sum_{t \ne j,l} f(l,j,s,t)\\
  &=  \sum_{j=1}^rf(j,j,j,j)  + 4  \sum_{l, j\ne l }^rf(l,j,j,j)  + 3 \sum_{l=1}^r  \sum_{s \ne j} f(l,l,s,s)   +    4  \sum_{l=1}^r  \sum_{j \ne l } \sum_{s \ne j, l} f(l,j,s,s) 
\\
 &  \quad +  \sum_{l=1}^r \sum_{s \ne l}    \sum_{t \ne s,l} f(l,s,s,t)   +  \sum_{l=1}^r \sum_{s \ne l}  \sum_{j \ne l, s }   f(l,j,s,s)   +  \sum_{l=1}^r \sum_{s \ne l}  \sum_{j \ne l, s }  \sum_{t \ne j,l, s } f(l,j,s,t)  
 \\
  &=  \sum_{j=1}^rf(j,j,j,j)  + 4  \sum_{l, j\ne l }^rf(l,j,j,j)  + 3 \sum_{l=1}^r  \sum_{s \ne j} f(l,l,s,s)   +    6  \sum_{l=1}^r  \sum_{j \ne l } \sum_{s \ne j, l} f(l,j,s,s) 
\\
 &   \quad+  \sum_{l=1}^r \sum_{s \ne l}  \sum_{j \ne l, s }  \sum_{t \ne j,l, s } f(l,j,s,t)  . 
  \end{align*} 
}
}} 

\end{document}